\documentclass[12pt,a4paper]{amsart}
\usepackage[top=1.15in, bottom=1.15in, left=1.17in, right=1.17in]{geometry}
\usepackage{graphicx}
\usepackage{amssymb}
\usepackage{epstopdf}
\usepackage{pdflscape}
\usepackage{amscd}
\usepackage{dsfont}
\usepackage{fancyhdr}
\usepackage{xypic}
\usepackage{color}

\newcommand{\Cc}{\mathcal{C}}
\newcommand{\ji}{{J(\mathcal{L})}}

\newcommand{\one}{{1\hskip-2.5pt\hbox{\rm I}}}

\newcommand{\ra}{\rightarrow}
\newcommand{\ve}{\varepsilon}
\newcommand{\vp}{\varphi}

\newcommand{\mi}{\mathbf{I}}
\newcommand{\ml}{\mathcal{L}}

\newcommand{\mc}{\mathbb{C}}

\newcommand{\FFLV}{{\operatorname*{FFLV}}}
\newcommand{\Gr}{{\operatorname*{Gr}}}
\newcommand{\SL}{{\operatorname*{SL}}}
\newcommand{\spec}{{\operatorname*{Spec}}}
\newcommand{\heit}{{\operatorname*{ht}}}
\newcommand{\gr}{{\operatorname*{gr}}}

\newcommand{\mSpec}{{\operatorname*{maxSpec}}}
\newcommand{\wwt}{{\operatorname*{wt}}}

\newtheorem{theorem}{Theorem}[section]
\newtheorem{corollary}[theorem]{Corollary}
\newtheorem{lemma}[theorem]{Lemma}
\newtheorem{proposition}[theorem]{Proposition}

\newtheorem{example}[theorem]{Example}
\newtheorem{definition}[theorem]{Definition}
\newtheorem{remark}[theorem]{Remark}

\title[Newton-Okounkov bodies and semi-toric degenerations]{From standard monomial theory to semi-toric degenerations via Newton-Okounkov bodies}
\author{Xin Fang and Peter Littelmann}
\address{\newline Mathematisches Institut, Universit\"at zu K\"oln, 50931, Cologne, Germany} 
\email{xfang@math.uni-koeln.de}
\address{\newline Mathematisches Institut, Universit\"at zu K\"oln, 50931, Cologne, Germany} 
\email{peter.littelmann@math.uni-koeln.de}
\dedicatory{Dedicated to Ernest Vinberg on the occasion of his 80th birthday}
\begin{document}
\maketitle

\begin{abstract}
The Hodge algebra structures on the homogeneous coordinate rings of Grassmann varieties provide semi-toric degenerations of these varieties. In this paper we construct these semi-toric degenerations using quasi-valuations and triangulations of Newton-Okounkov bodies.
\end{abstract}

\section{Introduction}
The basic idea of this paper is to test out in the simplest (but nontrivial) case - the Grassmann variety - how to combine ideas from standard monomial theory and associated semi-toric degenerations  \cite{Chi,DEP, S2} together with the theory of 
Newton-Okounkov bodies \cite{KK} and its associated toric degenerations \cite{A}.
\par
The study of flat degenerations of partial flag varieties started essentially with the work of Hodge \cite{Hodge}. 
There are in general two parallel directions in the study of these degenerations: the special fibre is a toric variety 
or a (reduced) union of toric varieties.
\par
In the first situation, many important developments in representation theory and discrete geometry, such as canonical bases, 
cluster algebras and the theory of Newton-Okounkov bodies, are applied to provide new insights in constructing different toric degenerations,
see  \cite{A,AB,Ca1,FaFL1,GHKK,GL,K1}, for details on the (incomplete) history, see for example \cite{FaFL2}. 
\par
Flat toric degenerations whose special fibres are no longer irreducible but a union of toric varieties are called semi-toric degenerations. 
The quest for semi-toric degenerations arises naturally for example in case one is looking for degenerations which are compatible with 
certain prescribed subvarieties: a typical example for such a situation are Schubert varieties in a Grassmann variety 
(a nice argument why in this example one needs semi-toric degenerations can be found in \cite{Ca1}). Semi-toric
degenerations occur naturally in the work of De Concini, Eisenbud and Procesi \cite{DEP} on Hodge algebras. 
In the case of partial flag varieties, such degenerations are constructed by Chiriv\`{i} \cite{Chi} using Lakshimibai-Seshadri (LS) 
algebra structures arising from the study of standard monomial theory of partial flag varieties \cite{L}.
\par
We strongly believe that the theory of standard monomials is connected to the theory of Newton-Okounkov bodies via triangulations
of the bodies. To make this vague statement more concrete, let us explain the picture we get in the case of the Grassmann variety.
\par
The combinatorial structure connected to standard monomial theory is controlled by a partially ordered set (for short we write just poset). 
In the case of the Grassmann variety $\Gr_{d,n}$, this is the set $I(d,n)$ of subsets of size $d$ of $\{1,\ldots,n\}$, with the partial order given by componentwise comparison.
\par
Let $R=\mathbb C[\Gr_{d,n}]$ be the homogeneous
coordinate ring given by the Pl\"ucker embedding $\Gr_{d,n}\hookrightarrow \mathbb P(\Lambda^d\mathbb C^n)$. 
For a given maximal chain $\mathcal C$ in the poset $I(d,n)$, we define a valuation $\nu_{\mathcal C}$ 
on the field of rational functions $\mathbb C(\Gr_{d,n})$, such that the associated Newton-Okounkov body $\mathcal P$ is, up to unimodular equivalence, independent of the choice of the chain. In fact, $\mathcal P$ is the so called Gelfand-Tsetlin polytope. 
Moreover, if one looks just at the values of the standard monomials 
that have support on the fixed chain $\mathcal C$, this defines a simplex embedded in $\mathcal P$. Indeed, by varying the maximal
chains, one gets a triangulation of $\mathcal P$ such that the simplexes are in bijection with the maximal chains. 
\par
To lift this triangulation up to the level of the Grassmann variety, we pass from the set of valuations 
$\{\nu_{\mathcal C}\mid \mathcal C\ {\rm a\ maximal\ chain}\}$ to a quasi-valuation \cite{KM} $\nu$ by taking the minimum of them: 
$$
\nu: \mathbb C(\Gr_{d,n})\setminus\{0\}\rightarrow \mathbb Z^N,\quad h\mapsto\min\{\nu_{\mathcal C}(h)\mid \mathcal{C}\hbox{\rm\ is a\ maximal\ chain}\}.
$$
This quasi-valuation induces a $\mathbb{Z}^N$-filtration $\mathcal F_\nu$ of $R$, such that the associated graded algebra ${\rm assgrad}_{\mathcal F_\nu} R$
is the discrete Hodge algebra \cite{DEP} associated to the poset $I(d,n)$. In other words, we have recovered the semi-toric degeneration 
of $\Gr_{d,n}$ into a union of $\mathbb P^N$'s described in \cite{DEP}.
\par
A geometric interpretation of the results described above is given by associating to each valuation $\nu_{\mathcal C}$ a toric degeneration, which is compatible with those Schubert varieties corresponding to the elements of the chain $\mathcal{C}$.
Therefore, by passing from a family of valuations to a quasi-valuation one only gets a semi-toric degeneration, but this has the advantage of being compatible with all Schubert varieties in $\Gr_{d,n}$. 
\par
The paper is organised as follows: after recalling basic notions and constructions on distributive (order) lattices and the associated Hibi varieties in Section \ref{Sec:Lattice} and \ref{Hibi}, we study valuations and quasi-valuations on Hibi varieties in Section \ref{Sec:Quasival} and \ref{Sec:QuasivalHibi}. In particular, we construct three different families of quasi-valuations on Hibi varieties and then apply them to construct semi-toric degenerations. The notion of an algebra governed by a lattice is introduced in Section \ref{Sec:nontoric}, and is applied to generalise the results on Hibi varieties to varieties that can be degenerated to Hibi varieties. In Section \ref{Sec:Grass} we show that Grassmann varieties fall into this category, and the previous constructions, once applied to these varieties, recover the Hodge algebra degeneration of Grassmann varieties. Relations to Feigin-Fourier-Littelmann-Vinberg polytopes are observed in Section \ref{Sec:FFLV}. In Section \ref{Sec:Outlook} we discuss questions and further directions of this work.

\section{Distributive lattices}\label{Sec:Lattice}

Let $(\ml,\vee,\wedge)$ be a finite bounded distributive lattice with operations {\it join} $\vee$ and {\it meet} $\wedge$. This structure induces a partial order on $\ml$ by:
$$p\le q\text{ if }p\wedge q=p.$$ 
With this partial order, $(\ml,\le)$ is a poset. For $p,q\in\ml$, $p$ is called a \textit{decent} of $q$ if $p<q$ and there exists no element $\ell$ in $\ml$ such that $p<\ell<q$. The unique minimal (resp. maximal) element in $\ml$ will be denoted by $\mathbb O$ (resp. $\one$). 
\par
Linearly ordered subsets in $\ml$ are called \textit{chains}. A chain $\Cc$ is called \textit{maximal} if for any other chain $\Cc'$, $\Cc\subseteq \Cc'$ implies $\Cc=\Cc'$. Let $C(\ml)$ denote the set of all maximal chains in $\ml$. The {\it length} $\mathrm{len}(\Cc)=\vert\Cc\vert-1$ of a chain $\Cc$ is the number of steps in the chain.
For a systematical introduction to lattice theory, see for example, \cite{Gr}.
\par
An element $m\in\ml$ is called \textit{join-irreducible} if $m=\ell_1\vee\ell_2$ for some $\ell_1,\ell_2\in \ml$ implies 
$m=\ell_1$ or $m=\ell_2$. Denote by $\ji$ the set of join-irreducible elements in $\ml$. 
The partial order on $\ml$ induces a partial order on $\ji$, making the latter a poset.
\par
Let $\mathcal P(\ji)$  be the {\it power set} of $\ji$, which is itself a lattice with the union of sets ``$\cup$'' as join operator and the intersection of sets ``$\cap$'' as meet operator.
A nonempty subset $\mathbf{b} \in  \mathcal P(\ji)$ is called an {\it order ideal} with respect to the induced partial order on $\ji$ if for all $m,m'\in \ji$ holds: $m\in \mathbf{b}$ and $m'<m$ implies $m'\in \mathbf{b}$.
Let $\mathcal{D}(\ji)\subset \mathcal P(\ji)$ be the set of subsets consisting of order ideals with respect to the partial order.  
\par
Two lattices are called isomorphic, if there exists a bijection between them preserving the join and meet operations. Two distributive lattices are isomorphic if they are isomorphic as lattices. Notice that endowed with the operations $\cup$ and $\cap$, $(\mathcal{D}(\ji),\cap,\cup)$ is a distributive lattice. The following theorem can be found in \cite[Theorem 107]{Gr}.

\begin{theorem}[Birkhoff]\label{Birkhoff}
The distributive lattices $(\mathcal{D}(\ji),\cap,\cup)$ and $(\ml,\vee,\wedge)$ are isomorphic.
\end{theorem}

The isomorphism in Birkhoff's theorem can be made explicit as follows: for $\ell\in \ml$ we define
$$
\spec(\ell)=\{m\in \ji\mid m\le \ell\},
$$
and let $\mSpec(\ell)$ be the set of maximal elements in $\spec(\ell)$.
\par
The following map provides the isomorphism in the theorem of Birkhoff:
$$
\ml\rightarrow \mathcal{D}(\ji),\quad \ell \mapsto \spec(\ell),
$$
whose inverse is given by:
$$ \mathcal{D}(\ji)\ra\ml,\quad \mathbf{b} \mapsto \bigvee_{m\in \mathbf{b}} m.$$

In the following we often identify the lattice $\ml$ with the lattice $\mathcal{D}(\ji)$. The length of a maximal chain in $\ml$ is equal to the cardinality of $\ji\setminus\{\mathbb{O}\}$.
\par
An enumeration $\ji=\{m_0,m_1,\ldots,m_N\}$ of the join-irreducible elements is called   
an {\it order preserving enumeration} if $m_i<m_j$
implies $i<j$. Let $E(\ml)$ be the set of all order preserving enumerations of $J(\ml)$. 
\par
We define a map $\vp:C(\ml)\ra E(\ml)$ as follows: starting with a maximal chain $\mathcal{C}$ in $\ml$, say
$$
\mathcal{C}: \mathbb O=c_0<c_1<c_2< \ldots< c_N =\one,
$$
we associate to $\mathcal{C}$ an enumeration of $J(\ml)$ by letting 
$$m_0=\mathbb{O}\text{ and for } i=1,\ldots,N,\ \ m_i\in\spec(c_i)\setminus \spec(c_{i-1})$$ 
be the unique new element.
This defines an order preserving enumeration. Conversely, given an order preserving enumeration $\{m_0=\mathbb{O}, m_1,\ldots,m_N\}$, the associated sequence of elements
$$
\Cc:\ m_0<m_1<(m_1\vee m_2)< \ldots<\mathop{\bigvee}_{1\le i\le j}m_i<\ldots < (m_1\vee m_2\vee\cdots\vee m_N) =\one
$$
is a maximal chain in $\ml$. 

Another immediate consequence of the isomorphism between 
$\ml$ and $\mathcal{D}(\ji)$ is: 
\begin{lemma}\label{chainLB}
The map $\vp:C(\ml)\ra E(\ml)$ is a bijection.
\end{lemma}


\section{The Hibi variety $\mathbb X_\ml$}\label{Hibi}
As before, let $\ml$ be a finite bounded distributive lattice. The associated Hibi variety \cite{Hibi} (or rather its 
projective version) is the variety $\mathbb X_\ml\subset \mathbb P(\mathbb C^{\vert \ml\vert})$ defined as the zero
set of the homogeneous ideal 
$$
I(\ml)=\langle X_{\ell_1}X_{\ell_2}-X_{\ell_1\wedge\ell_2}X_{\ell_1\vee\ell_2}\mid \ell_1,\ell_2\in \ml\text{ non-comparable}\rangle
\subset \mathbb C[X_\ell\mid \ell\in \ml].
$$
The homogeneous coordinate ring $R(\ml):= \mathbb C[X_\ell\mid \ell\in \ml] / I(\ml)$ is called the Hibi ring
of the lattice $\ml$. Since $I(\ml)$ is homogeneous, $R(\ml)$ is naturally endowed with a grading.

We write $x_\ell$ for the image of $X_\ell$ in $R(\ml)$. It is known that $\mathbb X_\ml$ is an irreducible, projectively normal embedded 
toric variety, and $R(\ml)$ is Cohen-Macaulay \cite{Hibi}. In addition, $R(\ml)$ is
an algebra with straightening law in the sense of De Concini, Eisenbud and Procesi \cite{DEP}. This implies
in particular that $R(\ml)$ has as $\mathbb C$-vector space a basis given by {\it standard monomials}, i.e., 
monomials of the form
$$
x_{\ell_1}x_{\ell_2}\cdots x_{\ell_r}\hbox{\rm\ where\ }\ell_1\ge \ell_2\ge\ldots\ge \ell_r.
$$
Denote by $\mathbb C(\mathbb X_\ml)$ the field of rational functions on $\mathbb X_\ml$, an element in $\mathbb C(\mathbb X_\ml)$ 
can always be represented
as a quotient $\frac{f}{g}$, where $f,g\in R(\ml)$ are homogeneous of the same degree.

In \cite{Hibi} one finds a second description of $R(\ml)$. Fix an order preserving enumeration
$\ji=\{m_0=\mathbb{O},m_1,\ldots,m_N\}$ of the join-irreducible elements
and identify $\ml$ with the set of order ideals $\mathcal{D}(\ji)$ (Theorem \ref{Birkhoff}).
Consider the polynomial ring $S_N=\mathbb C[y_0,\ldots,y_N]$, and let $M_{\ji}\subset S_N$ be the subset of monomials
$$
M_\ji =\{\prod_{m_i\in \spec(\ell)}y_i\mid \ell \in \ml\}.
$$
We denote by $A(\ml)$ the subalgebra of $S_N$ generated by the monomials in $M_\ji$. 
We endow $S_N$ with a grading by setting $\deg y_0=1$ and $\deg y_j=0$ for all $j\ge 1$.
The generators of $A(\ml)$ are homogeneous of degree 1, so $A(\ml)$ inherits in a natural way 
the structure of a graded algebra.
\par
As a graded algebra, $R(\ml)$ is isomorphic to $A(\ml)$ (\cite{Hibi}), the isomorphism is given on the generators by
$$
x_\ell \mapsto \prod_{m_i\in \spec(\ell)}y_i.
$$
This isomorphism provides an explicit description of the field of rational functions on $\mathbb X_\ml$:
\begin{lemma}\label{yisomo}
Let $\phi:\mathbb C(\mathbb X_\ml)\rightarrow \mathbb C(y_1,\ldots,y_N)$ be the map defined by: for $f,g\in A(\ml)$ homogeneous of the same degree,
$$\frac{f}{g}\mapsto \frac{f}{y_0^{\deg f}}/\frac{g}{y_0^{\deg g}}\in \mathbb C(y_1,\ldots,y_N).$$
Then $\phi$ is a field isomorphism.
\end{lemma}
\begin{proof}
An element in $\mathbb C(\mathbb X_\ml)$  can always be represented
as a quotient $\frac{f}{g}$, where $f,g\in R(\ml)\simeq A(\ml)$ are homogeneous of the same degree.
In terms of the ring $A(\ml)$ this means $f$ and $g$ are divisible by the same power of $y_0$ and hence
$\phi(\frac{f}{g})\in \mathbb C(y_1,\ldots,y_N)$. If $\frac{f}{g}=\frac{p}{q}$, then 
$$
\frac{f}{y_0^{\deg f}}/\frac{g}{y_0^{\deg g}}=\frac{p}{y_0^{\deg p}}/\frac{q}{y_0^{\deg q}},
$$
so the image is independent of the choice of the representative.
It follows that $\phi$ is well-defined and $\phi(\mathbb C(\mathbb X_\ml))\subseteq \mathbb C(y_1,\ldots,y_N)$.

Now one easily checks that $\phi$ is a ring homomorphism. 
Since the enumeration is order preserving, we have for all $i\ge 1$: $y_0\cdots y_{i-1}y_i$ and $y_0\cdots y_{i-1}$ are homogeneous elements
in $M_\ji$ of the same degree, 
and hence $y_i=\phi(\frac{y_0\cdots y_{i-1}y_i}{y_0\cdots y_{i-1}})\in \phi(\mathbb C(\mathbb X_\ml))$, which implies
$\phi(\mathbb C(\mathbb X_\ml))= \mathbb C(y_1,\ldots,y_N)$.
\end{proof}

Let $\Cc$ be a maximal chain in $\ml$. By Lemma~\ref{chainLB}, this can be identified with an order preserving enumeration of $\ji$. If $\Cc=\{\mathbb O=c_0<c_1<\ldots<c_N=\one\}$ 
and the corresponding enumeration is $\{m_0=\mathbb{O},m_1,\ldots,m_N\}$, then $c_1=m_1$, $c_2=m_1\vee m_2,\ldots$,
and the set of monomials associated to the elements in the chain are:
$$
M_\Cc:=\{x_{\mathbb O}=y_0,\ x_{c_1}=y_0y_1,\ x_{c_2}=y_0y_1y_2,\ \ldots,\ x_{c_N}=y_0y_1\cdots y_N\}.
$$
This description of $M_\Cc$ implies that the subalgebra $\mathbb C[\Cc]$ of $R(\ml)\simeq A(\ml)$ generated by 
$M_\Cc$ is isomorphic to a polynomial algebra. For $\ell\in\ml\setminus\{\mathbb{O}\}$, we denote
$$\hat{x}_\ell:=\prod_{m_i\in\spec(\ell)}\frac{x_{c_i}}{x_\mathbb{O}}\in\mathbb C(\mathbb X_\ml),$$
and $\hat{x}_\mathbb{O}=x_{\mathbb{O}}$.
We associate to $M_\Cc$ a sequence of rational functions
$$
\hat M_\Cc:=\{\hat x_{c_1}=y_1,\ \hat x_{c_2}=y_1y_2,\ \ldots,\ \hat  x_{c_N}=y_1\cdots y_N\}\subset \mathbb C(\mathbb X_\ml).
$$
We get as an immediate consequence:
\begin{corollary}\label{chainxisomo}
$\mathbb C(\mathbb X_\ml) = \mathbb C(\hat x_{c_1},\ldots,\hat x_{c_N})$.
\end{corollary}


\section{$\mathbb Z^N$-valued valuations and quasi-valuations}\label{Sec:Quasival}

\subsection{Valuations on function fields}\label{Sec:Val}
Let $X\subset \mathbb P^N$ be a projective variety with field of rational functions $\mathbb{C}(X)$ and homogeneous coordinate ring $\mathbb{C}[X]$. The notions of pre-valuations, valuations and quasi-valuations are available in general situations \cite{KK,KM}. We will restrict ourselves in this paper to these notions defined on the field $\mathbb C(X)$.
\par
By a {\it lexicographic type total order}  ``$\ge$'' on  $\mathbb Z^N$
we mean that ``$\ge$'' is either the lexicographic order or the reverse lexicographic order (see for example Chapter 2, Section 2 in \cite{CLO}). We fix such a total order ``$>$" on  $\mathbb Z^N$ and  write $(\mathbb Z^N,>)$
to emphasize that  $\mathbb Z^N$ is endowed with a fixed total order. 
\par
A $\mathbb Z^N$-valued \textit{pre-valuation} on $\mathbb C(X)$ is a map
$\nu: \mathbb C(X)\setminus\{0\}\rightarrow \mathbb Z^N$
such that 
\begin{itemize}
\item[{\it (a)}] $\nu(f+g)\ge \min\{\nu(f),\nu(g)\}$ for all nonzero $f$ and $g$ in $\mathbb C(X)$,
\item[{\it (b)}] $\nu(cf)= \nu(f)$ for all nonzero $f$ and $c\in\mathbb C^*$.
\end{itemize}
A pre-valuation $\nu:\mathbb{C}(X)\setminus\{0\}\ra\mathbb{Z}^N$ is called a \emph{valuation} if it satisfies the following condition (c); it is called a \textit{quasi-valuation} if it satisfies the following condition (c'):
\begin{itemize}
\item[{\it (c)}] $\nu(fg)=\nu(f)+\nu(g)$ for all nonzero $f$ and $g$.
\item[{\it (c')}] $\nu(fg)\geq\nu(f)+\nu(g)$ for all nonzero $f$ and $g$.
\end{itemize}

Let $\nu$ be a quasi-valuation. For $\mathbf{v}\in\mathbb{Z}^N$ we define
$$
\nu_{\geq\mathbf{v}}:=\{f\in \mathbb C(X)\setminus\{0\}\mid \nu(f)\ge \mathbf{v}\}\cup\{0\},\ \ \nu_{>\mathbf{v}}:=\{f\in \mathbb C(X)\setminus\{0\}\mid \nu(f)> \mathbf{v}\}\cup\{0\}.
$$
The associated {\it leaf} is defined to be the quotient vector space
$$
\nu_\mathbf{v}:=\nu_{\geq\mathbf{v}}/\nu_{>\mathbf{v}}.
$$
We say that $\nu$ has at most {\it one-dimensional leaves} if $\dim \nu_\mathbf{v}\le 1$ for all $\mathbf{v}\in \mathbb Z^N$.
\par
Restricting $\nu$ to the subalgebra $\mathbb{C}[X]$ gives a $\mathbb{Z}^N$-algebra filtration on $\mathbb{C}[X]$.

\begin{proposition}\label{minquasivalue}
Let $X\subset \mathbb P^N$ be a projective variety and let 
$\nu_1,\ldots,\nu_r$ be $\mathbb Z^N$-valued quasi-valuations on $\mathbb C(X)$. Set
$$
\nu: \mathbb C(X)\setminus\{0\}\rightarrow \mathbb Z^N,\quad h\mapsto\min\{\nu_j(h)\mid j=1,\ldots,r\}.
$$
Then $\nu$ is a $\mathbb Z^N$-valued quasi-valuation.
\end{proposition}
\begin{proof} Let $f,g\in \mathbb C(X)$, then
$$
\begin{array}{rcl}
\nu(f+g)&=&\min\{\nu_j(f+g)\mid  j=1,\ldots,r\}\\
&\ge&\min\{\nu_j(f),\nu_j(g)\mid  j=1,\ldots,r\}\\
&=&\min\{\nu(f),\nu(g)\}
\end{array}
$$
for all nonzero $f$ and $g$ in $\mathbb C(X)$. Multiplying by a nonzero complex number does not change
the value of the quasi-valuations, and
$$
\begin{array}{rcl}
\nu(fg)&=&\min\{\nu_j(fg)\mid  j=1,\ldots,r\}\\
&\ge&\min\{\nu_j(f)\mid  j=1,\ldots,r\} +\min\{\nu_j(g)\mid  j=1,\ldots,r\}\\
&=&\nu(f)+\nu(g).
\end{array}
$$
It follows that $\nu$ satisfies the conditions {\it a)}, {\it b)} and {\it c')}, and hence $\nu$ is a quasi-valuation. 
\end{proof}

\subsection{Valuations for Hibi varieties}\label{valuationsmatrices}
Fix a maximal chain $\Cc=\{\mathbb{O}=c_0<c_1<\ldots<c_N=\one\}$ in $\ml$. Since $\mathbb C(\mathbb X_\ml)= \mathbb C(\hat x_{c_1},\ldots,\hat x_{c_N})$ by 
Corollary~\ref{chainxisomo}, a given $\mathbb Z^N$-valued valuation on $\mathbb C(\mathbb X_\ml)$
is completely determined by its values on the generators. So one can attach to $\Cc$ 
a matrix $B_{\nu,\Cc}\in \mathcal{M}_N(\mathbb Z)$ having as columns the values of $\nu$ on the generators (see also \cite{KM}):
$$
(\nu,\Cc)\mapsto B_{\nu,\Cc}=\bigg(\nu(\hat x_{c_1}) \vert \ldots \vert\nu(\hat x_{c_N})\bigg).
$$
If $\nu$ has at most one-dimensional leaves, then the columns of this matrix are $\mathbb Q$-linearly independent.
Let $B\in \mathcal{M}_N(\mathbb Z)$ be such that $\det B\not=0$, and let $\mathbf{v}_1,\ldots,\mathbf{v}_N$ be the column vectors. 
We define a valuation on $\mathbb C(\mathbb X_\ml)=\mathbb C(\hat x_{c_1},\ldots,\hat x_{c_N})$ as follows.
\par
We use the abbreviation ${\hat x}^{\mathbf{n}}$ for $\hat x_{c_1}^{n_1}\cdots \hat x_{c_N}^{n_N}$. For a monomial ${\hat x}^{\mathbf{n}}\in\mathbb{C}[\hat x_{c_1},\ldots,\hat x_{c_N}]$, we define
$$
\nu_{B,\Cc}({\hat x}^{\mathbf{n}}):=\sum_{j=1}^N n_j \mathbf{v}_j\in\mathbb{Z}^N;
$$
for polynomials in $\mathbb{C}[\hat x_{c_1},\ldots,\hat x_{c_N}]$, we define
$$
\nu_{B,\Cc}(\sum_{\mathbf{n}} c_{\mathbf{n}} {\hat x}^{\mathbf{n}}):=\min\{\nu_{B,\Cc}({\hat x}^{\mathbf{n}})\mid c_{\mathbf{n}}\not=0\}.
$$
By construction, $\nu_{B,\Cc}$ satisfies the conditions {\it a)} and {\it b)} in Section \ref{Sec:Val}. Moreover, $\nu_{B,\Cc}$ is additive on the product of monomials, \emph{i.e.}  
$$
\nu_{B,\Cc}({\hat x}^{\mathbf{n}}{\hat x}^{\mathbf{q}})=\nu_{B,\Cc}({\hat x}^{\mathbf{n}+\mathbf{q}})=
\nu_{B,\Cc}({\hat x}^{\mathbf{n}})+\nu_{B,\Cc}({\hat x}^{\mathbf{q}}).
$$
Since a lexicographic-type order has been fixed on $\mathbb Z^N$, we have in addition the following property: if $\nu_{B,\Cc}(x^{\mathbf{n}})>\nu_{B,\Cc}(x^{\mathbf{p}})$,
then for any $\mathbf{q}\in\mathbb{Z}^N$,
\begin{equation}\label{additivity}
\nu_{B,\Cc}({\hat x}^{\mathbf{n}+\mathbf{q}})=
\nu_{B,\Cc}({\hat x}^{\mathbf{n}})+\nu_{B,\Cc}({\hat x}^{\mathbf{q}})>
\nu_{B,\Cc}({\hat x}^{\mathbf{p}})+\nu_{B,\Cc}({\hat x}^{\mathbf{q}})=\nu_{B,\Cc}({\hat x}^{\mathbf{p}+\mathbf{q}}).
\end{equation}
It follows that $\nu_{B,\Cc}(fg)= \nu_{B,\Cc}(f)+ \nu_{B,\Cc}(g)$ for all $f,g\in \mathbb C[\hat x_{c_1},\ldots,\hat x_{c_N}]$. Extending the map $\nu_{B,\Cc}$ to the fraction field by setting 
$$\nu_{B,\Cc}\left(\frac{f}{g}\right)=\nu_{B,\Cc}(f)-\nu_{B,\Cc}(g),$$
we obtain a valuation on $\mathbb{C}(X)$. The linear independence of the column vectors of $B$
implies that $\nu_{B,\Cc}$ is a valuation with at most one-dimensional leaves.


\section{Examples for quasi-valuations on Hibi varieties}\label{Sec:QuasivalHibi}

We provide in this section some examples of quasi-valuations on Hibi varieties.
\par
As before, let $\ml$ be a finite bounded distributive lattice. By the \textit{height} $\heit(\ell)$ of an element $\ell\in \ml$ we mean the length of a chain joining $\ell$ with the unique minimal element.
\par
Let $\{\mathbf{e}_1,\ldots,\mathbf{e}_N\}$ (resp. $\{\mathbf{e}_0,\mathbf{e}_1,\ldots,\mathbf{e}_N\}$) be the standard basis of $\mathbb Z^N$ (resp. $\mathbb{Z}^{N+1}$). For this section we fix as total order on $\mathbb Z^{N}$ 
the {\it reverse lexicographic order}.  On $\mathbb Z^{N+1}$ we fix the {\it graded reverse lexicographic order}, where the degree
is provided by the coefficient of $\mathbf e_0$.

\subsection{The support quasi-valuation}\label{supportmap}
We fix a maximal chain  $\Cc=\{\mathbb O=c_0<c_1<\ldots<c_N=\one\}$ in $\ml$, in order to identify $\mathbb C(\mathbb X_\ml)$ with $\mathbb C(\hat x_{c_1},\ldots,\hat x_{c_N})$.
Let $\ji=\{m_0=\mathbb{O}, m_1,m_2,\ldots,m_N\}$ be the associated order preserving enumeration of the join-irreducible elements. For $\ell\in\ml$, let $\spec(\ell)^*=\spec(\ell)\setminus\{\mathbb{O}\}$.
By the arguments in Section~\ref{valuationsmatrices}, the map which associates to $\hat x_{c_j}$ the vector 
$$\sum_{m_i\in \spec(c_j)^*}\mathbf{e}_{i}\in\mathbb Z^N$$
can be extended to a valuation
$$
\nu_{\Cc,\spec}:\mathbb C(\mathbb X_\ml)\setminus\{0\}\rightarrow \mathbb Z^N.
$$
\par
Let now $\mathbb X_\ml\subset  \mathbb P(\mathbb C^{\vert \ml\vert})$ be the embedded Hibi variety and
denote by $R(\ml)=\bigoplus_{i\geq 0} R_i$ the homogeneous coordinate ring (see Section~\ref{Hibi}). 
We use the valuation $\nu_{\Cc,\spec}$ to define the valuation monoid
associated to $R(\ml)$ by
$$
\Gamma_{\nu_{\Cc,\spec}}(R(\ml))= \bigoplus_{i\ge 0}\Gamma_{\nu_{\Cc,\spec}}(R_i), 
\hbox{\rm\ where\ \ } \Gamma_{\nu_{\Cc,\spec}}(R_i)=\{i\mathbf{e}_0 + \nu_{\mathcal{C},\spec}\left(\frac{h}{x_0^i}\right) \mid h\in R_i\}.
$$
The associated Newton-Okounkov body is the closure of the convex hull:
$$
\mathrm{NO}_{\mathcal C}(\ml)=\overline{{\rm conv}\left({\bigcup_{j\geq 1}\left\{ \frac{1}{j} v \mid v\in \Gamma_{\nu_{\Cc,\spec}}(R_j) \right\}}\right)}.
$$

By Proposition \ref{minquasivalue}, we define the quasi-valuation $\nu_{\spec}$ as follows:
$$
\nu_\spec:\mathbb C(\mathbb X_L)\setminus\{0\}\rightarrow \mathbb Z^N, \quad \nu_{\spec}(h):=\min\{\nu_{\Cc,\spec}(h)\mid \Cc\in C(\ml)\}.
$$
\begin{theorem}\label{Thm:support}
Let $\ell_1,\ldots,\ell_k\in\ml$ and $n_1,\ldots,n_k$ be non-zero natural numbers.
The quasi-valuation $\nu_{\spec}:\mathbb C(\mathbb X_\ml)\setminus\{0\}\rightarrow \mathbb Z^N$ satisfies the following properties: 
\begin{enumerate}
\item For any maximal chain $\mathcal{C}$, 
$$
\nu_{\spec}(\hat x_{\ell_1}^{n_1}\cdots \hat x_{\ell_k}^{n_k})\le \nu_{\Cc,\spec}(\hat x_{\ell_1}^{n_1}\cdots \hat x_{\ell_k}^{n_k}),
$$
the equality holds only if $\{\ell_1,\ldots,\ell_k\}\subset \Cc$.
\item If the monomial $\hat x_{\ell_1}^{n_1}\cdots \hat x_{\ell_k}^{n_k}$ is standard, and $\mathcal C$ is a maximal chain 
containing  $\{\ell_1,\ldots,\ell_k\}$, then 
$$
\nu_{\spec}(\hat x_{\ell_1}^{n_1}\cdots \hat x_{\ell_k}^{n_k})=
\sum_{j=1}^k n_j \nu_{\spec}(\hat x_{\ell_j})
=\sum_{j=1}^k n_j \nu_{\mathcal C,\spec}(\hat x_{\ell_j}).
$$
\item If the monomial $\hat x_{\ell_1}^{n_1}\cdots \hat x_{\ell_k}^{n_k}$ is not standard, then
$$
\nu_{\spec}(\hat x_{\ell_1}^{n_1}\cdots \hat x_{\ell_k}^{n_k})
>\sum_{j=1}^k n_j \nu_{\spec}(\hat x_{\ell_j}).
$$
\end{enumerate}
\end{theorem}

\begin{proof}
For a fixed maximal chain $\Cc=\{\mathbb O=c_0<c_1<\ldots<c_N=\one\}$ we have
by definition:
$$
\nu_{\Cc,\spec}(\hat x_{c_j})=\mathbf{e}_j +\sum_{i=1}^{j-1}\mathbf{e}_i.
$$
Recall that the index $j$ of $\mathbf{e}_j$ is also the height of $c_j$. If $\ell\not\in \Cc$ and 
$x_\ell$ corresponds to $\prod_{m_i\in \spec(\ell)}y_i=y_{i_1}\cdots y_{i_r}$, where $1\le i_1<\ldots<i_r\le N$, hence
$$
\hat x_\ell = \frac{{\hat x}_{c_{i_1}}}{{\hat x}_{c_{i_1-1}}}\cdots \frac{{\hat x}_{c_{i_r}}}{{\hat x}_{c_{i_r-1}}}.
$$
This presentation is not unique, there might be cancellations, but the term $\hat x_{c_{i_r}}$ shows up in the nominator and not in the denominator
in any presentation. It follows that 
$$
\nu_{\Cc,\spec}(\hat x_\ell)=\mathbf{e}_{i_r}+\sum_{j=1}^{i_r-1}\lambda_{j}\mathbf{e}_j,
$$
where $\lambda_j\in\{0,1\}$. Now $\ell=\bigvee_{m\in \spec(\ell)} m$ and $\spec(\ell)\subset \{m_1,\ldots,m_{i_r}\}$, so it follows
that $c_{i_r}=m_1\vee m_2\vee\ldots\vee m_{i_r}$ is larger or equal to $\ell$. By assumption, $\ell\not\in \Cc$, so we have $c_{i_r}>\ell$
and hence $i_r=\heit(c_{i_r})>\heit(\ell)$. By the definition of $\heit(\ell)$, there exists another maximal chain $\mathcal{C}'$ such that $\nu_{\mathcal{C}',\spec}(\hat x_\ell)<\nu_{\mathcal{C},\spec}(\hat x_\ell)$.
\par
We prove the statement (1): the first statement holds by definition. If $\{\ell_1,\ldots,\ell_k\}$ is not contained in $\Cc$, then
there exists a smallest $s$ such that $\ell_s\notin\Cc$. The argument above shows that $
\nu_{\spec}(\hat x_{\ell_1}^{n_1}\cdots \hat x_{\ell_k}^{n_k})< \nu_{\Cc,\spec}(\hat x_{\ell_1}^{n_1}\cdots \hat x_{\ell_k}^{n_k}).$
\par
To prove the statement (2), notice that $\hat x_{\ell_1}^{n_1}\cdots \hat x_{\ell_k}^{n_k}$ is a standard monomial  implies that $\ell_1\geq\ell_2\geq\cdots\geq\ell_k$. We extend it to a maximal chain $\Cc$ in $\ml$ and apply the first part of the theorem.
\par
For the statement (3),  $\hat x_{\ell_1}^{n_1}\cdots \hat x_{\ell_k}^{n_k}$ is not a standard monomial implies that there is no maximal chain containing all $\ell_1,\ldots,\ell_k$, so we have
$$
\nu_{\spec}(\hat x_{\ell_1}^{n_1}\cdots \hat x_{\ell_k}^{n_k}) = \min\left\{\sum_{j=1}^kn_j\nu_{\Cc,\spec}(\hat x_{\ell_j})\mid \Cc\in C(\ml)\right\}
> \sum_{j=1}^kn_j\nu_{\spec}(\hat x_{\ell_j}).
$$
\end{proof}

\subsection{The maximal support quasi-valuation}\label{maxSpecmap}
For $\ell\in\ml$, instead of the entire $\spec(\ell)$, one can use the maximal elements $\mSpec(\ell)$ to define a family of valuations and a quasi-valuation. Fix a maximal chain $\Cc=\{\mathbb O=c_0<c_1<\ldots<c_N=\one\}$ in $\ml$. Let $\ji=\{m_0=\mathbb{O}, m_1,m_2,\ldots,m_N\}$ be the associated order preserving enumeration of the join-irreducible elements. Let $\nu_{\Cc,\mSpec}$ be the map associating to $\hat x_{c_j}$ for $j=1,\ldots,N$ the vector 
$$\sum_{m_i\in \mSpec(c_j)}\mathbf{e}_{i}\in\mathbb Z^N.$$ 
By the argument in Section~\ref{valuationsmatrices}, it can be extended to a $\mathbb Z^N$-valued valuation
$$
\nu_{\Cc,\mSpec}:\mathbb C(\mathbb X_\ml)\setminus\{0\}\rightarrow \mathbb Z^N.
$$
By Proposition \ref{minquasivalue}, we define the quasi-valuation $\nu_{\mSpec}$ as follows:
$$
\nu_\mSpec:\mathbb C(\mathbb X_L)\setminus\{0\}\rightarrow \mathbb Z^N, \quad \nu_{\mSpec}(h):=\min\{\nu_{\Cc,\mSpec}(h)\mid \Cc\in C(\ml)\}.
$$
The proof of the following theorem is similar to that of Theorem \ref{Thm:support}.

\begin{theorem}\label{Thm:maxSpec}
Let $\ell_1,\ldots,\ell_k\in\ml$ and $n_1,\ldots,n_k$ be non-zero natural numbers.
The quasi-valuation $\nu_{\mSpec}:\mathbb C(\mathbb X_\ml)\setminus\{0\}\rightarrow \mathbb Z^N$ satisfies the following properties: 
\begin{enumerate}
\item For any maximal chain $\mathcal{C}$, 
$$
\nu_{\mSpec}(\hat x_{\ell_1}^{n_1}\cdots \hat x_{\ell_k}^{n_k})\le \nu_{\Cc,\mSpec}(\hat x_{\ell_1}^{n_1}\cdots \hat x_{\ell_k}^{n_k}),
$$
the equality holds only if $\{\ell_1,\ldots,\ell_k\}\subset \Cc$.
\item If the monomial $\hat x_{\ell_1}^{n_1}\cdots \hat x_{\ell_k}^{n_k}$ is standard, and $\mathcal C$ is a maximal chain containing  $\{\ell_1,\ldots,\ell_k\}$, then 
$$
\nu_{\mSpec}(\hat x_{\ell_1}^{n_1}\cdots \hat x_{\ell_k}^{n_k})=
\sum_{j=1}^k n_j \nu_{\mSpec}(\hat x_{\ell_j})
=\sum_{j=1}^k n_j \nu_{\mathcal C,\mSpec}(\hat x_{\ell_j}).
$$
\item If the monomial $\hat x_{\ell_1}^{n_1}\cdots \hat x_{\ell_k}^{n_k}$ is not standard, then
$$
\nu_{\mSpec}(\hat x_{\ell_1}^{n_1}\cdots \hat x_{\ell_k}^{n_k})
>\sum_{j=1}^k n_j \nu_{\mSpec}(\hat x_{\ell_j}).
$$
\end{enumerate}
\end{theorem}

\subsection{The height quasi-valuation}\label{heightmap}
One can construct another quasi-valuation using the height. Fix a maximal chain $\Cc=\{\mathbb O=c_0<c_1<\ldots<c_N=\one\}$ in $\ml$ and let $\nu_{\Cc,\heit}$ be the map associating to $\hat x_{c_j}$ for $j=1,\ldots,N$ the vector $\mathbf{e}_{j}$. By the argument in Section~\ref{valuationsmatrices}, it can be extended to a $\mathbb Z^N$-valued valuation
$$
\nu_{\Cc,\heit}:\mathbb C(\mathbb X_\ml)\setminus\{0\}\rightarrow \mathbb Z^N.
$$
By Proposition \ref{minquasivalue}, we define the quasi-valuation $\nu_{\heit}$ as follows:
$$
\nu_\heit:\mathbb C(\mathbb X_L)\setminus\{0\}\rightarrow \mathbb Z^N, \quad \nu_{\heit}(h):=\min\{\nu_{\Cc,\heit}(h)\mid \Cc\in C(\ml)\}.
$$
The proof of the following theorem is similar to that of Theorem \ref{Thm:support}.

\begin{theorem}\label{Thm:height}
Let $\ell_1,\ldots,\ell_k\in\ml$ and $n_1,\ldots,n_k$ be non-zero natural numbers.
The quasi-valuation $\nu_{\heit}:\mathbb C(\mathbb X_\ml)\setminus\{0\}\rightarrow \mathbb Z^N$ satisfies the following properties: 
\begin{enumerate}
\item For any maximal chain $\mathcal{C}$, 
$$
\nu_{\heit}(\hat x_{\ell_1}^{n_1}\cdots \hat x_{\ell_k}^{n_k})\le \nu_{\Cc,\heit}(\hat x_{\ell_1}^{n_1}\cdots \hat x_{\ell_k}^{n_k}),
$$
the equality holds only if $\{\ell_1,\ldots,\ell_k\}\subset \Cc$.
\item If the monomial $\hat x_{\ell_1}^{n_1}\cdots \hat x_{\ell_k}^{n_k}$ is standard, and $\mathcal C$ is a maximal chain containing  $\{\ell_1,\ldots,\ell_k\}$, then 
$$
\nu_{\heit}(\hat x_{\ell_1}^{n_1}\cdots \hat x_{\ell_k}^{n_k})=
\sum_{j=1}^k n_j \nu_{\heit}(\hat x_{\ell_j})
=\sum_{j=1}^k n_j \nu_{\mathcal C,\heit}(\hat x_{\ell_j}).
$$
\item If the monomial $\hat x_{\ell_1}^{n_1}\cdots \hat x_{\ell_k}^{n_k}$ is not standard, then
$$
\nu_{\heit}(\hat x_{\ell_1}^{n_1}\cdots \hat x_{\ell_k}^{n_k})
>\sum_{j=1}^k n_j \nu_{\heit}(\hat x_{\ell_j}).
$$
\end{enumerate}
\end{theorem}

\subsection{Applications to semi-toric degenerations}

Let $\nu$ be one of the quasi-valuations $\nu_{\spec}$, $\nu_{\mSpec}$ or $\nu_{\heit}$ defined on $\mc(\mathbb X_{\ml})\setminus\{0\}$ above.
\par
Restricting $\nu$ to the homogeneous coordinate ring $R(\ml)=\mc[\mathbb{X}_\ml]$ gives a $\mathbb{Z}^N$-filtration of algebra on $R(\ml)$. Let $\gr_\nu(R(\ml))$ denote the associated graded algebra. The following corollary is a consequence of Theorem \ref{Thm:support}, \ref{Thm:maxSpec} and \ref{Thm:height}.

\begin{corollary}\label{Cor:discrete}
The graded algebra $\gr_\nu(R(\ml))$ is the algebra generated by $X_\ell$ with $\ell\in\ml$ and the following relations:
$$\text{if $\ell_1$ and $\ell_2$ are not comparable in $\ml$, then}\ X_{\ell_1}X_{\ell_2}=0.$$
Moreover, the images of standard monomials in $R(\ml)$ form a basis of $\gr_\nu(R(\ml))$.
\end{corollary}

Using standard arguments (see for example \cite{DEP}), one obtains from this construction a flat 
degeneration of the projective toric variety $\mathbb{X}_\ml$ into a union of toric varieties,
such that each irreducible component is isomorphic to $\mathbb P^N$.


\section{A lift to some non-toric cases}\label{Sec:nontoric}
We want to extend the construction of the previous subsections to varieties, which are not necessarily toric varieties and to construct in this way semi-toric degenerations.
The following construction is inspired by the theory of Hodge algebras by De Concini, Eisenbud and Procesi \cite{DEP}.
Let  $X\subset \mathbb P(V)$ be an embedded projective variety with homogeneous 
coordinate ring $R=\bigoplus_{i\geq 0} R_i$. Let $\ml$ be a finite bounded distributive lattice and let
$\psi: \ml\rightarrow R_1$ an injective map of sets. We write $x_\ell$ for the image $\psi(\ell)\in R_1$.
\begin{definition}\label{straightlaw1}\rm
We say that $R$ is {\it governed by $\ml$}, if the set of {\it standard monomials} 
$$
{\rm SMon}:=\{x_{k_1}\cdots x_{k_r}\mid k_1\ge \ldots\ge k_r\in \ml, r\in \mathbb Z_{\ge 0}\}
$$ 
forms a vector space basis for $R$, and if $\ell_1,\ell_2\in \ml$ are not comparable and
\begin{equation}\label{straightlaw2}
x_{\ell_1}x_{\ell_2}= \sum_{k_1\ge k_2} a_{k_1,k_2} x_{k_1} x_{k_2}
\end{equation}
is the unique expression of $x_{\ell_1}x_{\ell_2}$ as a linear combination of standard monomials, then
\begin{enumerate}
\item[(a)] $a_{\ell_1\vee \ell_2,\ell_1\wedge \ell_2}=1$;
\item[(b)] if for some $(k_1,k_2)\neq (\ell_1\vee\ell_2,\ell_1\wedge\ell_2)$, $a_{k_1,k_2}\neq 0$, then for every pair $(m_1,m_2)$ where $m_1\in\mSpec(\ell_1\vee\ell_2)$ and $m_2\in\mSpec(\ell_1\wedge\ell_2)$ such that $m_1\geq m_2$, one of the following statements holds:
\begin{itemize}
\item there exists $h\in\mSpec(k_1)$ such that $h>m_1$; 
\item the statement above does not hold, and there exist $h\neq h'\in\mSpec(k_1)$ such that $h=m_1$ and $h'>m_2$;
\item the statements above do not hold, and there exist $h\in\mSpec(k_1)$, $h'\in\mSpec(k_2)$ such that $h=m_1$ and $h'\geq m_2$.
\end{itemize}
\end{enumerate}

\begin{remark}
Compared to the Hodge algebra defined in \cite{DEP}, some requirements in the notion of an algebra governed by a distributive lattice are stronger, for example: the relations in \eqref{straightlaw2} are quadratic; the leading coefficient $a_{\ell_1\vee \ell_2,\ell_1\wedge \ell_2}$ is $1$. However, the last two conditions in the part (b) of the above definition are not apparently comparable with the conditions in a Hodge algebra. Nevertheless, we expect that if $R$ is an algebra governed by a distributive lattice $\ml$, then $R$ admits a Hodge algebra structure generated by $\psi(\ml)$.
\end{remark}

\end{definition}
Fix a maximal chain  $\Cc=\{\mathbb O<c_1<\ldots<c_N=\one\}$ in $\ml$ and let $\ji=\{m_0=\mathbb{O},m_1,m_2,\ldots,m_N\}$ be 
the associated enumeration of the join-irreducible elements. We define a map from the set of standard monomials SMon to $\mathbb Z^{N+1}$
using the valuation $\nu_{\mathcal C,\spec}$ defined in Section~\ref{supportmap}. 
Let $\{\mathbf{e}_0,\mathbf{e}_1,\ldots,\mathbf{e}_N\}\subset  \mathbb Z^{N+1}$ be the canonical basis. We define
$$
\mu_{\mathcal C,\spec}:{\rm SMon}\rightarrow \mathbb Z^{N+1},\quad x{_\mathbb O}^{a_0}x_{c_1}^{a_1}\ldots x_{c_N}^{a_N}
\mapsto (\sum_{i=0}^N a_i)\,\mathbf{e}_0 + \nu_{\mathcal C,\spec}(\hat x_{c_1}^{a_1}\ldots \hat x_{c_N}^{a_N}).
$$
The coefficient of $\mathbf{e}_0$ is the total degree of the monomial.
\begin{theorem}
If $R$ is governed by $\ml$, then the map $\mu_{\mathcal C,\spec}$ extends to a valuation $\mu_{\mathcal C,\spec}:R\setminus\{0\}\rightarrow \mathbb Z^{N+1}$.
\end{theorem}
\begin{proof}
We begin with two remarks:
\begin{enumerate}
\item[(i)] There are two orders on $J(\ml)$:
\begin{itemize}
\item the induced partial order $>$ from $\ml$ (which is independent of the choice of $\Cc$);
\item choosing a maximal chain $\mathcal{C}$ in $\ml$ provides an order preserving enumeration of $J(\ml)$, by taking the  associated reverse lexicographic order we obtain the total order $\succ$ on it.
\end{itemize}
It is clear that $\succ$ is a refinement of $>$, \emph{i.e.,} for $\ell_1,\ell_2\in\ml$, $\ell_1>\ell_2$ implies $\ell_1\succ\ell_2$.
\item[(ii)] Given $\ell_1,\ell_2\in\ml$ and $\mathcal{C}\in C(\ml)$, for the valuation on the field of rational functions of the Hibi variety we have:
$$\nu_{\mathcal C,\spec}(\hat x_{\ell_1}\hat x_{\ell_2})= \nu_{\mathcal C,\spec}(\hat x_{\ell_1\vee \ell_2}\hat x_{\ell_1\wedge\ell_2}).$$
\end{enumerate}

The map $\mu_{\mathcal{C},\spec}$ is defined on the linear basis $\mathrm{SMon}$ of $R$, we extend the map to linear combinations by taking the minimum:
$$
\mu_{\mathcal C,\spec}\left(\sum_{\mathbf m\in{\rm SMon}} c_{\mathbf m} \mathbf m\right)
=\min\{  \mu_{\mathcal C,\spec}(\mathbf m)\mid c_{\mathbf m}\not=0\}.
$$
This defines a pre-valuation on $R$. The valuation $ \nu_{\mathcal C,\spec}$ is defined on all monomials,
so for two non-comparable elements $\ell_1,\ell_2\in \ml$ (notice that none of them is equal to $\mathbb O$) we set
\begin{equation}\label{munu}
\mu_{\mathcal C,\spec}(x_{\ell_1}x_{\ell_2})=2\mathbf{e}_0+  \nu_{\mathcal C,\spec}(\hat x_{\ell_1}\hat x_{\ell_2}).
\end{equation}
We have to show that this convention does not contradict the definition via sums of standard monomials given before:
$$
\begin{array}{rcl}
\mu_{\mathcal C,\spec}(x_{\ell_1}x_{\ell_2})&=&\mu_{\mathcal C,\spec}(\sum_{k_1\ge k_2} a_{k_1,k_2} x_{k_1} x_{k_2})\\
&=&\min\{ \mu_{\mathcal C,\spec}(x_{k_1} x_{k_2})\mid  a_{k_1,k_2}\not=0  \}\\
&=&\min\{ 2\mathbf{e}_0+\nu_{\mathcal C,\spec}(\hat x_{k_1} \hat x_{k_2})\mid  a_{k_1,k_2}\not=0  \}
\end{array}
$$
by the formula \eqref{straightlaw2}. 
By the second remark above, it remains to show that
$$
\nu_{\mathcal C,\spec}(\hat x_{k_1}\hat x_{k_2})> \nu_{\mathcal C,\spec}(\hat x_{\ell_1\vee \ell_2}\hat x_{\ell_1\wedge\ell_2})
$$
for all other terms showing up in \eqref{straightlaw2} with nonzero coefficients. 
\par
Let $k_1,k_2\in\ml$ be two different elements such that $a_{k_1,k_2}\not=0$ in \eqref{straightlaw2} and $(k_1,k_2)\not=(\ell_1\vee \ell_2,\ell_1\wedge \ell_2)$. 

We first assume that $n_1\in\mSpec(\ell_1\vee\ell_2)$ is in addition the maximal element with respect to the total order $\succ$, and $n_2\in\mSpec(\ell_1\wedge\ell_2)$ is furthermore the maximal element with respect to $\succ$ among those elements $n$ such that $n_1\succeq n$. If there exists $h\in\mSpec(k_1)$ such that $h>n_1$, then $h\succ n_1$ and hence $$\nu_{\mathcal C,\spec}(\hat x_{k_1}\hat x_{k_2})> \nu_{\mathcal C,\spec}
(\hat x_{\ell_1\vee \ell_2}\hat x_{\ell_1\wedge\ell_2}).$$ 
\par
Suppose that this is not the case, then by Definition~\ref{straightlaw1}, there exists $h_1\in\mSpec(k_1)$ and $h_2$ in either $\mSpec(k_1)$ (in this case $h_1\neq h_2$) or $\mSpec(k_2)$, such that $h_1=n_1$ and $h_2\geq n_2$.
\par
If one can find an $n_1'\in \spec(\ell_1\vee\ell_2)$ such that
$n_1'\succ n_2$, then proceed with the pair $(n_1',n_2')$ satisfying: $n'_1\in \mSpec(\ell_1\vee\ell_2)$ 
and $n'_2\in \mSpec(\ell_1\wedge\ell_2)$ such that $n'_1\succ n'_2$,
and $n'_1\prec n_1$ is maximal with this property, and  $n'_2$ 
is maximal with respect to ``$\succ$" among those $n$ satisfying $n'_1\ge  n$.
\par
If such an $n_1'$ can not be found, then for any $m\in \spec(\ell_1\vee\ell_2)\setminus\{n_1,n_2\}$, we have $n_2\succ m$. There exist two possibilities: if $h_2> n_2$, then $h_2\succ n_2$ and hence 
again 
$$\nu_{\mathcal C,\spec}(\hat x_{k_1}\hat x_{k_2})> \nu_{\mathcal C,\spec}
(\hat x_{\ell_1\vee \ell_2}\hat x_{\ell_1\wedge\ell_2}).$$
If the equality $h_2= n_2$ holds, then one proceeds
with the next pair $(n_1',n_2')$, where
$n'_1\in \mSpec(\ell_1\vee\ell_2)$ and $n'_2\in \mSpec(\ell_1\wedge\ell_2)$ such that $n'_1> n'_2$,
and $n'_1\prec n_1$ is maximal with this property, and  $n'_2$ 
is maximal with respect to ``$\succ$" among those elements $n$ satisfying $n'_1\ge n$.

Since the maximal elements completely determine the value of $\nu_{\mathcal C,\spec}$, one obtains inductively 
that after a finite number of steps, either
$$\nu_{\mathcal C,\spec}(\hat x_{k_1}\hat x_{k_2})> \nu_{\mathcal C,\spec}
(\hat x_{\ell_1\vee \ell_2}\hat x_{\ell_1\wedge\ell_2})$$
or
$$\nu_{\mathcal C,\spec}(\hat x_{k_1}\hat x_{k_2})=\nu_{\mathcal C,\spec}(\hat x_{\ell_1\vee \ell_2}\hat x_{\ell_1\wedge\ell_2})$$
holds.
But the latter only occurs when all the maximal elements of $k_1$ and 
$\ell_1\vee \ell_2$ agree, which can only happen when $k_1=\ell_1\vee \ell_2$ and $k_2=\ell_1\wedge\ell_2$.

Since every non-standard monomial can be rewritten in a finite number of steps 
using \eqref{straightlaw2} into a linear combination
of standard monomials, applying \eqref{munu} in each step shows that if we define
$$
\mu_{\mathcal C,\spec}(x_{c_1}\ldots x_{c_r}):=r\mathbf{e}_0+  \nu_{\mathcal C,\spec}(\hat x_{c_1}\ldots \hat x_{c_r})
$$
then this coincides with the value of $ \mu_{\mathcal C,\spec}$ on the minimal term in the expression of the monomial
in terms of standard monomials. This implies that if we have two sums of standard monomials,
then the product is a priori not anymore a sum of standard monomials, but the value of $\mu_{\mathcal C,\spec}$
on the sum is the value of $\mu_{\mathcal C,\spec}$ on the product of the two minimal summands. 
It follows that $\mu_{\mathcal C,\spec}$ is a $\mathbb Z^{N+1}$-valued valuation.
\end{proof}

Now one can proceed as in the case of the Hibi variety: in the same way as in Proposition~\ref{minquasivalue},
one shows that the map
$$
\mu_{\spec}: R\setminus\{0\}\rightarrow \mathbb Z^{N+1},\quad h\mapsto\min\{ \mu_{\Cc,\spec}(h)\mid \mathcal C
\in C(\ml)\}.
$$
is a $\mathbb Z^{N+1}$-valued quasi-valuation.
\par
The proof of the following theorem is similar to Theorem~\ref{Thm:support}.

\begin{theorem}\label{Thm:Govern}
Let $\ell_1,\ldots,\ell_k\in\ml$ and $n_1,\ldots,n_k$ be non-zero natural numbers.
The quasi-valuation $\mu_{\spec}:R\setminus\{0\}\rightarrow \mathbb Z^{N+1}$ satisfies the following properties: 
\begin{enumerate}
\item For any maximal chain $\mathcal{C}$, 
$$
\mu_{\spec}(x_{\ell_1}^{n_1}\cdots x_{\ell_k}^{n_k})\le \mu_{\Cc,\spec} (x_{\ell_1}^{n_1}\cdots x_{\ell_k}^{n_k}),
$$
the equality holds only if $\{\ell_1,\ldots,\ell_k\}\subset \Cc$.
\item If the monomial $ x_{\ell_1}^{n_1}\cdots x_{\ell_k}^{n_k}$ is standard, and $\mathcal{C}$ is a maximal chain  containing  $\{\ell_1,\ldots,\ell_k\}$, then 
$$
\mu_{\spec}(x_{\ell_1}^{n_1}\cdots x_{\ell_k}^{n_k})=
\sum_{j=1}^k n_j \mu_{\spec}( x_{\ell_j})
=\sum_{j=1}^k n_j \mu_{\mathcal C,\spec}( x_{\ell_j}).
$$
\item If the monomial $x_{\ell_1}^{n_1}\cdots x_{\ell_k}^{n_k}$ is not standard, then
$$
\mu_{\spec}( x_{\ell_1}^{n_1}\cdots x_{\ell_k}^{n_k})
>\sum_{j=1}^k n_j \mu_{\spec}(x_{\ell_j}).
$$
\end{enumerate}
\end{theorem}

By applying Theorem \ref{Thm:Govern}, a similar result to Corollary \ref{Cor:discrete} can be proved for algebras $R$ governed by $\ml$.
\par
The quasi-valuation $\mu_\spec$ induces a $\mathbb{Z}^{N+1}$-filtration of algebra on $R$, we let $\gr_\mu(R)$ denote the associated graded algebra.

\begin{corollary}\label{Cor:discretegovern}
The graded algebra $\gr_\mu(R)$ is the algebra generated by $x_\ell$ for $\ell\in\ml$ and the following relations:
$$\text{if $\ell_1$ and $\ell_2$ are not comparable in $\ml$, then}\ x_{\ell_1}x_{\ell_2}=0.$$
Moreover, the images of standard monomials in $R$ form a basis of $\gr_\mu(R)$.
\end{corollary}


\section{Applications to Grassmann varieties}\label{Sec:Grass}

The results of the last sections will be applied to study semi-toric degenerations of Grassmann varieties.

\subsection{The distributive lattice $I(d,n)$}\label{Sec:Idn}
Let $n\geq 1$ be a positive integer. For $1\leq d\leq n$ we define
$$I(d,n):=\{\mathbf{I}=[i_1,\ldots,i_d]\mid 1\le i_1<\ldots<i_d\le n\}.$$ 
The {\it meet} and {\it join} operations are defined by
$$
\mi\wedge \mathbf{J}:=[\min\{i_1,j_1\},\ldots,\min\{i_d,j_d\}],\quad
\mi\vee \mathbf{J}:=[\max\{i_1,j_1\},\ldots,\max\{i_d,j_d\}].
$$
These operations make $I(d,n)$ into a finite distributive lattice, the induced partial order on $I(d,n)$ is exactly the following natural order: 
$$
\mi=[i_1,\ldots,i_d]\ge \mathbf{J}=[j_1,\ldots,j_d]\quad\hbox{\ if and only if\ }\quad i_1\ge j_1,\ldots, i_d\ge j_d.
$$
The lattice $I(d,n)$ is bounded with a unique minimal element $\mathbb O=[1,2,\ldots,d]$ and a unique maximal element $\one:=[n-d+1,\ldots,n]$. 
\par
For $\ml=I(d,n)$, recall that $J(\ml)$ is the set of join-irreducible elements in $\ml$. The elements of $J(\ml)$ can be divided into two families:
\begin{enumerate}
\item the consecutive family: for $k=0,\ldots,n-d$, $\mi_{0,k}=[k+1,k+2,\ldots,k+d]$;
\item the one descent family: for $1\leq s\leq d-1$ and $t>s-1$, $\mi_{s,t}=[1,\ldots,s,t+1,\ldots,t+d-s]$.
\end{enumerate}

The sub-poset $J(\ml)$ of $\ml$ looks like a block which can be presented in the following way (see for example \cite{BL}):
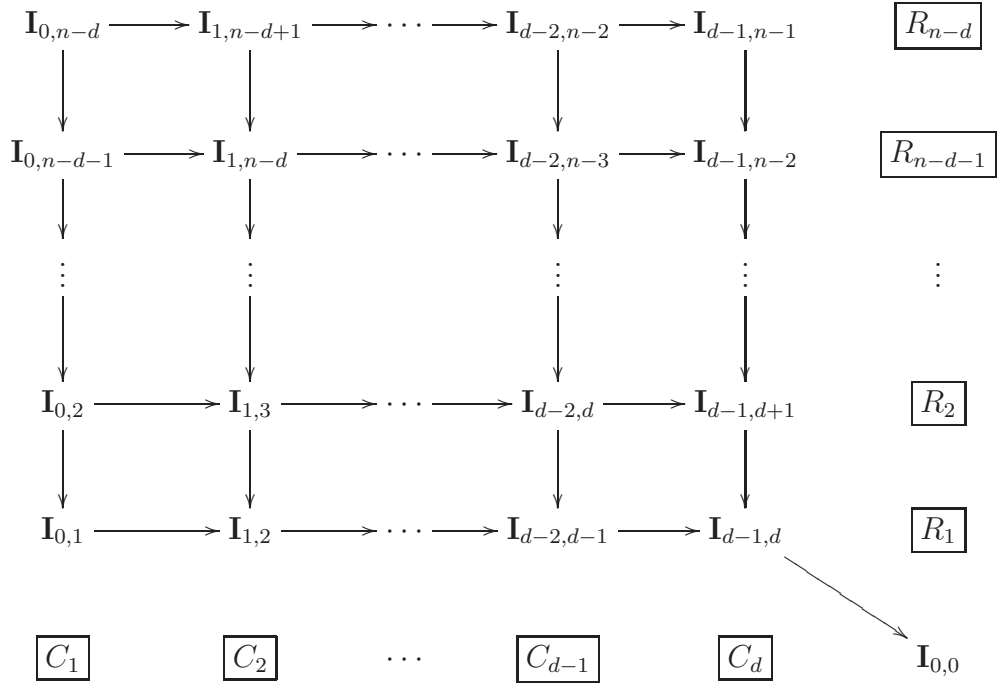
$$
\xymatrix{
\mi_{0,n-d} \ar[r] \ar[d] & \mi_{1,n-d+1} \ar[r] \ar[d] & \cdots \ar[r] & \mi_{d-2,n-2} \ar[r] \ar[d] & \mi_{d-1,n-1} \ar[d] & \fbox{$R_{n-d}$}\\
\mi_{0,n-d-1} \ar[r] \ar[d] & \mi_{1,n-d} \ar[r] \ar[d] & \cdots \ar[r] & \mi_{d-2,n-3} \ar[r] \ar[d] & \mi_{d-1,n-2} \ar[d] & \fbox{$R_{n-d-1}$}\\
\vdots \ar[d]& \vdots \ar[d]& & \vdots \ar[d]& \vdots \ar[d]& \vdots \\
\mi_{0,2} \ar[r] \ar[d] & \mi_{1,3} \ar[r] \ar[d] & \cdots \ar[r] & \mi_{d-2,d} \ar[r] \ar[d] & \mi_{d-1,d+1} \ar[d] & \fbox{$R_2$}\\
\mi_{0,1} \ar[r]  & \mi_{1,2} \ar[r] & \cdots \ar[r] & \mi_{d-2,d-1} \ar[r] & \mi_{d-1,d} \ar[dr] & \fbox{$R_1$}\\
\fbox{$C_1$}& \fbox{$C_2$}& \cdots  & \fbox{$C_{d-1}$}& \fbox{$C_d$}& \mi_{0,0}
}$$
In the diagram, arrows stand for the descents, and $C_1,\ldots,C_d,R_1,\ldots,R_{n-d}$ are the corresponding columns and rows:
$$C_k=\{\mi_{k-1,k+1},\mi_{k-1,k+2},\ldots,\mi_{k-1,k+d}\},$$
$$R_k=\{\mi_{0,k},\mi_{1,k+1},\ldots,\mi_{d-1,k+d-1}\}.$$

\begin{example}\label{Ex:I47}
We provide in this example $J(\ml)$ in the case $d=4$ and $n=7$.

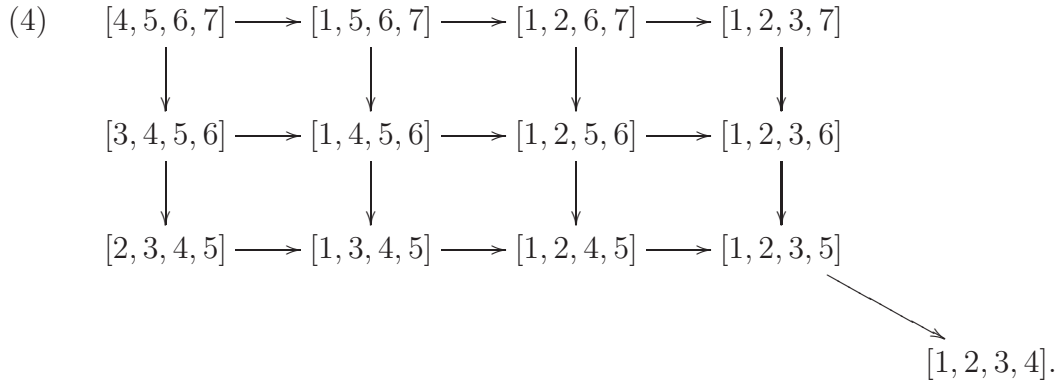
\begin{equation}\label{joinirred47}
\xymatrix{
\ar[d][4,5,6,7]\ar[r]&\ar[d][1,5,6,7]\ar[r]&\ar[d][1,2,6,7]\ar[r]&\ar[d][1,2,3,7] &
\\
\ar[d][3,4,5,6]\ar[r]&\ar[d][1,4,5,6]\ar[r]&\ar[d][1,2,5,6]\ar[r]&\ar[d][1,2,3,6] &
\\
        [2,3,4,5]\ar[r]&         [1,3,4,5]\ar[r]& [1,2,4,5]\ar[r]& [1,2,3,5]\ar[dr] &
\\
& & & & [1,2,3,4].
}
\end{equation}
We order the join-irreducible elements
in $J(\ml)$ in a rectangle as in \eqref{joinirred47}. An element $\mi\in I(d,n)$
corresponds in this picture to a subset of $J(\ml)$ below a staircase (mounting from left to right),
for example $\mi=[2,4,5,7]$ corresponds to
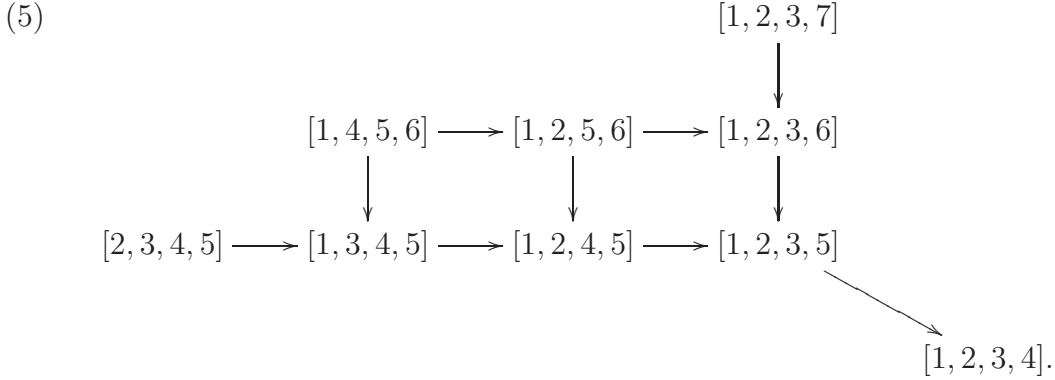
\begin{equation}\label{joinirred47subset}
\xymatrix{
& & &\ar[d][1,2,3,7] &\\
&\ar[d][1,4,5,6]\ar[r]&\ar[d][1,2,5,6]\ar[r]&\ar[d][1,2,3,6] &
\\
[2,3,4,5]\ar[r]&[1,3,4,5]\ar[r]& [1,2,4,5]\ar[r]& [1,2,3,5] \ar[dr] &
\\
& & & & [1,2,3,4].
}
\end{equation}
\end{example}

There exists a weight structure on $I(d,n)$. We fix a basis $\ve_1,\ve_2,\ldots,\ve_n$ of $\mathbb{R}^n$. For $1\leq i\leq j\leq n$, let $\alpha_{i,j}=\ve_{j+1}-\ve_i$ and $\alpha_i=\alpha_{i,i}$. Then $\alpha_1,\ldots,\alpha_{n-1}$ is a basis of $H=\{\sum_{i=1}^n x_i\ve_i\in\mathbb{R}^n\mid \sum_{i=1}^n x_i=0\}$. The \textit{weight} of an element $[i_1,\ldots,i_d]\in I(d,n)$ is given by:
$$\wwt([i_1,\ldots,i_d])=\ve_{i_1}+\ve_{i_2}+\ldots+\ve_{i_d}.$$

We define a map $\omega:J(\ml)\ra\mathbb{R}^n$ as follows: $\omega(\mi_{0,0})=0$ and for $(s,t)\neq (0,0)$,
$$\omega(\mi_{s,t})=\ve_{t+1}-\ve_t.$$
The map $\omega$ induces a map $\omega:\mathcal{P}(J(\ml))\ra \mathbb{R}^n$: for a subset $S$ of $J(\ml)$, 
$$\omega(S):=\sum_{\mi\in S}\omega(\mi).$$

We attach to the set $J(\ml)^*:=J(\ml)\setminus\{\mathbb{O}\}$ the following graph $S_{d,n}$:
\begin{enumerate}
\item for each $\mi\in J(\ml)^*$, there exists a vertex in $S_{d,n}$ labelled by $\omega(\mi)$;
\item there exists an edge between two vertices if and only if one vertex is the descent of the other.
\end{enumerate} 

The graph $S_{d,n}$ can be presented as follows:

$$
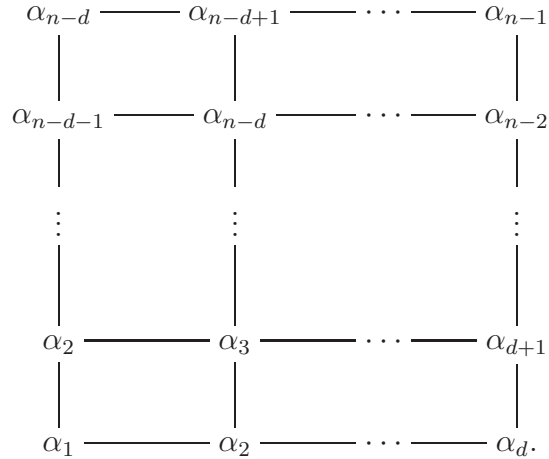
\xymatrix{
\alpha_{n-d} \ar@{-}[r] \ar@{-}[d] & \alpha_{n-d+1} \ar@{-}[r] \ar@{-}[d] & \cdots \ar@{-}[r] & \alpha_{n-1} \ar@{-}[d] \\
\alpha_{n-d-1} \ar@{-}[r] \ar@{-}[d] & \alpha_{n-d} \ar@{-}[r] \ar@{-}[d] & \cdots \ar@{-}[r] & \alpha_{n-2} \ar@{-}[d] \\
\vdots \ar@{-}[d]& \vdots \ar@{-}[d]&  &\vdots \ar@{-}[d] \\
\alpha_{2} \ar@{-}[r] \ar@{-}[d] & \alpha_{3} \ar@{-}[r] \ar@{-}[d] & \cdots \ar@{-}[r] & \alpha_{d+1} \ar@{-}[d] \\
\alpha_{1} \ar@{-}[r] & \alpha_{2} \ar@{-}[r] & \cdots \ar@{-}[r] & \alpha_{d}. 
}$$

For each ordered ideal $\mathbf{b}$ in $J(\ml)^*$, one can associate the full sub-graph $\mathcal{G}_\mathbf{b}$ in $S_{d,n}$ containing vertices corresponding to $\mi\in\mathbf{b}$.

\begin{lemma}
For $\mi\in\ml$, $\omega(\spec(\mi))=\wwt(\mi)-\wwt(\mi_{0,0})\in H$.
\end{lemma}

\begin{proof}
Let $\mi=[i_1,\ldots,i_d]$ with $i_1<\ldots<i_d$. 
\par
We claim that for any $t=1,2,\ldots,d$, $\omega(\spec(\mi)\cap C_t)=\ve_{i_t}-\ve_t$. Indeed, $\spec(\mi)\cap C_t=\{\mi_{t-1,t},\ldots,\mi_{t-1,i_t-1}\}$, hence $\omega(\spec(\mi)\cap C_t)=\ve_{i_t}-\ve_t$.
\par
As $\spec(\mi)$ is the disjoint union of $\spec(\mi)\cap C_t$ for $t=1,\ldots, d$, we obtain:
$$\omega(\spec(\mi))=\ve_{i_1}+\ldots+\ve_{i_d}-\ve_1-\ldots-\ve_d=\wwt(\mi)-\wwt(\mi_{0,0}).$$
\end{proof}

\begin{example}\label{Ex:I472}
We continue Example \ref{Ex:I47} to study $I(4,7)$. In this case, the graph $S_{4,7}$ looks like
\begin{equation}\label{joinirredroot47}
\xymatrix{
\ar@{-}[d]\alpha_3\ar@{-}[r]&\ar@{-}[d]\alpha_4\ar@{-}[r]&\ar@{-}[d]\alpha_5\ar@{-}[r]&\ar@{-}[d]\alpha_6 
\\
\ar@{-}[d]\alpha_2\ar@{-}[r]&\ar@{-}[d]\alpha_3\ar@{-}[r]&\ar@{-}[d]\alpha_4\ar@{-}[r]&\ar@{-}[d]\alpha_5 
\\
        \alpha_1\ar@{-}[r]&         \alpha_2\ar@{-}[r]& \alpha_3\ar@{-}[r]& \alpha_4.
}
\end{equation}

Let $\mi=[2,4,5,7]$. The order ideal $\spec(\mi)$ is given in \eqref{joinirred47subset}. The corresponding sub-graph $\mathcal{G}_{\spec(\mi)}$ in the graph $S_{4,7}$ looks like
\begin{equation}\label{joinirred47subsetroot}
\xymatrix{
& & &\ar@{-}[d]\alpha_6\\
&\ar@{-}[d]\alpha_3\ar@{-}[r]&\ar@{-}[d]\alpha_4\ar@{-}[r]&\ar@{-}[d]\alpha_5
\\
\alpha_1\ar@{-}[r]&\alpha_2\ar@{-}[r]& \alpha_3\ar@{-}[r]& \alpha_4.
\\
}
\end{equation}
Summing up all roots in the above graph gives 
$$\alpha_1+\alpha_2+2\alpha_3+2\alpha_4+\alpha_5+\alpha_6=(\ve_2+\ve_4+\ve_5+\ve_7)-(\ve_1+\ve_2+\ve_3+\ve_4).$$
\end{example}

\subsection{Grassmann varieties}
For more details on Grassmann varieties, see for example \cite{LB}.
\par
Let $n\geq 1$ be an integer. For $1\leq d\leq n$, the Grassmannian $\Gr_{d,n}$ is the set of $d$-dimensional subspaces in $\mathbb{C}^n$.
The projective variety structure on $\Gr_{d,n}$ is given by the Pl\"ucker embedding $\Gr_{d,n}\hookrightarrow \mathbb P(\Lambda^d\mathbb C^n)$, sending a $d$-dimensional subspace $\operatorname{span}\{v_1,\ldots,v_k\}\subset\mathbb{C}^n$ to the point $[v_1\wedge\ldots\wedge v_k]\in\mathbb{P}(\Lambda^k\mathbb{C}^n)$. The homogeneous coordinate ring $R:=\mathbb C[\Gr_{d,n}]$ then inherits from the embedding a grading $R=\bigoplus_{i\geq 0} R_i$.
\par
Let $e_1,\ldots,e_n$ be the standard basis of $\mathbb{C}^n$. For $\mi=[i_1,\ldots,i_d]$ with $1\leq i_1<\ldots<i_d\leq n$, let $p_\mi\in(\Lambda^d\mathbb{C}^n)^*$ denote the dual basis of $e_{i_1}\wedge\ldots\wedge e_{i_d}$. These $p_\mi$ are called Pl\"ucker coordinates of $\Gr(d,n)$. 
\par
Let $1\leq t\leq n$, $\sigma\in\mathfrak{S}_{d+1}/({\frak S}_{t}\times{\frak S}_{d+1-t})$ be a shuffle, $\mi=[i_1,\ldots,i_d]$ and $\mathbf{J}=[j_1,\ldots,j_d]$. We define 
$$\mi^\sigma=(\sigma^{-1}(i_1),\ldots,\sigma^{-1}(i_t),i_{t+1},\ldots,i_d),$$ 
$$\mathbf{J}^\sigma=(j_1,\ldots,j_{t-1},\sigma^{-1}(j_t),\ldots,\sigma^{-1}(j_d)).$$
\par
The homogeneous ideal $I_{d,n}\subset \mathbb C[\,p_{\mi}\mid\mi\in I(d,n)]$
generated by the Pl\"ucker relations
\begin{equation}\label{quadstraight}
\left\{\sum_{\sigma\in {\frak S}_{d+1}/({\frak S}_{t}\times{\frak S}_{d+1-t})}
\hbox{\rm sign}(\sigma) p_{\mi^\sigma}p_{\mathbf{J}^\sigma} 
\mid \mathbf I,\mathbf J\in I(d,n),1\le t \le n \right\}
\end{equation}
defines the Pl\"ucker embedding $\Gr_{d,n}\hookrightarrow \mathbb P(\Lambda^d\mathbb C^n)$ of the Grassmann variety,
\emph{i.e.} the homogeneous coordinate ring $\mathbb C[\Gr_{d,n}]$ is
isomorphic to $\mathbb C[\,p_{\mathbf I}\mid \mathbf I\in I(d,n)]/I_{d,n}$ (see for example \cite[Section 1.3]{SMT}. 
\par
Another way of describing the homogeneous coordinate ring is using the Hodge algebra \cite{DEP, Hodge}. Let $\psi:I(d,n)\ra R_1$ be the map sending $\mi\in I(d,n)$ to the Pl\"ucker coordinate $p_\mi$.
 It is known that 
$\mathbb C[\Gr_{d,n}]$ has as basis the standard monomials, \textit{i.e.}, monomials of the form
$$
p_{\mi_1}p_{\mi_2}\cdots p_{\mi_r}
\hbox{\rm\ where\ }\mi_1\geq\mi_2\geq\ldots\geq\mi_r.
$$
If $\mi_1, \mi_2$ are not comparable, then the Pl\"ucker 
relations can be used to find an expression of the product 
\begin{equation}\label{straightening}
p_{\mi_1}p_{\mi_2}= p_{\mi_1\vee\mi_2}p_{\mi_1\wedge \mi_2}
+\sum_{\mathbf{K}_1>\mi_1\vee\mi_2>
\mi_1\wedge \mi_2>\mathbf K_2} 
a_{\mathbf K_1,\mathbf K_2}
p_{\mathbf K_1}p_{\mathbf K_2}
\end{equation}
as a linear combination of standard monomials of degree 2, where the coefficient $1$ in the leading term is provided by \cite[Lemma 7.32]{GL}.

\subsection{Root poset $R_{d,n}$}

We consider the following root poset $R_{d,n}$ for $\Gr_{d,n}$, realized as $\SL_n/P_d$ where $P_d$ is the maximal parabolic subgroup associated to the simple root $\alpha_d$:
$$
\xymatrix{
\alpha_{1,d} \ar[r] \ar[d] & \alpha_{2,d} \ar[r] \ar[d] & \cdots \ar[r] & \alpha_{d,d} \ar[d] \\
\alpha_{1,d+1} \ar[r] \ar[d] & \alpha_{2,d+1} \ar[r] \ar[d] & \cdots \ar[r] & \alpha_{d,d+1} \ar[d] \\
\vdots \ar[d]& \vdots \ar[d]&  &\vdots \ar[d] \\
\alpha_{1,n-2} \ar[r] \ar[d] & \alpha_{2,n-2} \ar[r] \ar[d] & \cdots \ar[r] & \alpha_{d,n-2} \ar[d] \\
\alpha_{1,n-1} \ar[r] & \alpha_{2,n-1} \ar[r] & \cdots \ar[r] & \alpha_{d,n-1}.
}$$
This poset will be used in Proposition \ref{Prop:GT} and Section \ref{Sec:FFLV}.

\subsection{Semi-toric degenerations of Grassmann varieties}

We first show that the algebra $R=\mathbb{C}[\Gr_{d,n}]$ is governed by $\ml=I(d,n)$.

\begin{proposition}
The homogeneous coordinate ring $\mathbb C[\Gr_{d,n}]$ of the Grassmann variety
is governed by the finite bounded distributive lattice $I(d,n)$.
\end{proposition}
\begin{proof}
We have to show that if $a_{\mathbf{K}_1,\mathbf{K}_2}\not=0$ in \eqref{straightening},
then for every $\mathbf{M}\in \mSpec(\mi_1\vee\mi_2)$,
\begin{itemize}
\item there exists $\mathbf{H}\in\mSpec(\mathbf{K}_1)$ such that $\mathbf{H}>\mathbf{M}$; 
\item if one can not find such an element,
then there exist two different elements $\mathbf{H},\mathbf{H}'\in\mSpec(\mathbf{K}_1)$ such that $\mathbf{H}=\mathbf{M}$
and $\mathbf{H}'> \mathbf{M}'$ for any maximal element $\mathbf{M}'\in  \mSpec(\mi_1\wedge\mi_2)$ which is smaller or equal to $\mathbf{M}$;
\item if such a pair does not exist, then there exist $\mathbf{H}\in\mSpec(\mathbf{K}_1)$, $\mathbf{H}'\in\mSpec(\mathbf{K}_2)$ such that $\mathbf{H}=\mathbf{M}$
and $\mathbf{H}'\ge \mathbf{M}'$ for any $\mathbf{M}'\in  \mSpec(\mi_1\wedge\mi_2)$ 
which is smaller or equal to $\mathbf{M}$.
\end{itemize}
\par
First notice that elements in the set $\mSpec(\mathbf{K}_1)$
correspond exactly to the corners of the staircase in the associated order ideal in $J(I(d,n))$, see \eqref{joinirred47subset} for an example.
We enumerate the maximal elements (or the corners) from right to left: in the above example the enumeration is given by
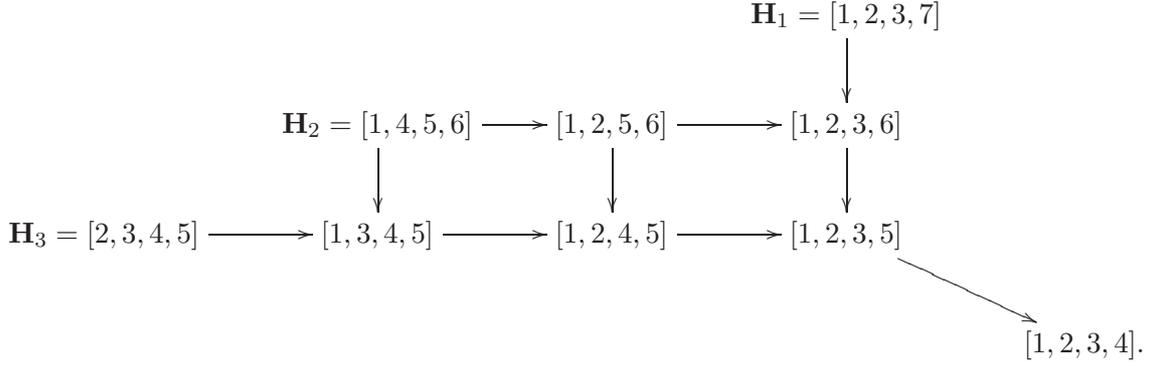
\begin{equation}\label{joinirred47subsetb}
\small{
\xymatrix{
                      &        &                              &    \mathbf{H}_1=\ar[d][1,2,3,7] &\\
                      &        \mathbf{H}_2=\ar[d][1,4,5,6]\ar[r]&\ar[d][1,2,5,6]\ar[r]&\ar[d][1,2,3,6] &
\\
\mathbf{H}_3=[2,3,4,5] \ar[r]&[1,3,4,5]                 \ar[r]& [1,2,4,5]\ar[r]& [1,2,3,5]\ar[dr] &
\\
& & & & [1,2,3,4].
}}
\end{equation}
Note that the elements in $\mSpec(\mi_1\vee\mi_2)$
correspond to the corners of its associated staircase, which lies below the staircase associated to $\spec(\mathbf{K}_1)$.
\par
Since $\mathbf{K}_1> \mi_1\vee\mi_2$ one has
$\spec(\mi_1\vee\mi_2)\subset \spec(\mathbf{K}_1)$, so for every
$\mathbf{M}\in \mSpec(\mi_1\vee\mi_2)$ there exists $\mathbf{H}\in\mSpec(\mathbf{K}_1)$ such that $\mathbf{H}\ge \mathbf{M}$. 
Suppose now one can find only such an $\mathbf{H}$ so that $\mathbf{H}=\mathbf{M}$. Having equality $\mathbf{H}= \mathbf{M}$ means that two staircases
share a common corner, so there exists a $j$ such that $\mathbf{H}=\mathbf{M}=\mathbf{H}_j$. Let $\mathbf{M}'\in  \mSpec(\mi_1\wedge\mi_2)$ 
be an element which is smaller or equal to $\mathbf{M}$. So $\mathbf{M}'$ lies in the staircase below and to the right of $\mathbf{H}_j$. 
If $\mathbf{H}_{j-1}$ or $\mathbf{H}_{j+1}$ exists and one of the two is strictly larger than $\mathbf{M}'$ (equality is not possible in this case since $\mathbf{M}'<\mathbf{M}$), then we are done. It remains to consider
the case where neither $\mathbf{H}_{j-1}>\mathbf{M}'$ nor $\mathbf{H}_{j+1}>\mathbf{M}'$ (if they exist). In this case $\mathbf{M}'$ lies in 
the rectangle formed by the columns where $\mathbf{M}=\mathbf{H}_{j}$ is an entry and the first column to the left of $\mathbf{H}_{j-1}$ 
(respectively the last row if $j=1$), and the rows containing $\mathbf{M}=\mathbf{H}_{j}$ respectively the row just above
$\mathbf{H}_{j+1}$ (respectively the bottom row if $\mathbf{H}_{j+1}$ does not exist). We have to find an element 
$\mathbf{H}'\in\spec(\mathbf{K}_2)$ such that $\mathbf{H}'\ge \mathbf{M}'$.
\par
Let $\alpha_i=\omega(\mathbf{M}')$. Since $\mathbf{H}=\mathbf{M}$ and because of the special location of $\mathbf{M}'$ in the rectangle 
described above, the $\alpha_i$-component of $\omega(\spec(\mathbf{K}_1))$ and $\omega(\spec(\mi_1\vee\mi_2))$
coincide. Now for weight reasons, the $\alpha_i$-component of $\omega(\spec(\mathbf{K}_2))$ and $\omega(\spec(\mi_1\wedge\mi_2))$
also have to coincide. But this implies that the staircase associated to $\mathbf{K}_2$ has to include
$\mathbf{M}'$. More precisely, $\mathbf{M}'$ has to be an element in the tread of a stairstep. So the next corner to the left of the staircase
is a maximal element $\mathbf{H}'\in \mathbf{K}_2$, which is larger or equal to $\mathbf{M}'$, finishing the proof.
\end{proof}

Then by results in Section \ref{Sec:nontoric}, for each maximal chain $\mathcal{C}\in C(\ml)$, we have the valuation $\mu_{\mathcal{C},\spec}$ on $\mathbb{C}(\Gr_{d,n})\setminus\{0\}$. By taking the minimum we obtain a quasi-valuation $\mu_{\spec}$ . We are at the point to apply Theorem \ref{Thm:Govern} and Corollary \ref{Cor:discretegovern}, as well as the construction of Hodge algebras in \cite{DEP}.

\begin{corollary}
 There exists a flat degeneration of $\Gr_{d,n}$ into a union of toric varieties,
such that the defining ideal of the initial scheme is generated by the monomials $p_{\mathbf I}p_{\mathbf J}$
for all non-comparable pairs $(\mathbf I, \mathbf J)$ in $\ml$. The initial scheme is a union of projective spaces, one for each maximal chain in $\ml$.
\end{corollary}

Moreover, for a maximal chain $\mathcal{C}$, we can identify the corresponding Newton-Okounkov body.

\begin{proposition}\label{Prop:GT}
For any maximal chain $\mathcal{C}\in C(\ml)$, the Newton-Okounkov body $\mathrm{NO}_{\mathcal C}(\ml)$ is unimodularly equivalent to the Gelfand-Tsetlin polytope.
\end{proposition}

\begin{proof}
The Gelfand-Tsetlin polytope associated to $\Gr(d,n)$ is by definition the order polytope  associated to the poset $R_{d,n}$. In the case of $R_{d,n}$, lattice points in the order polytopes are vertices, which are the characteristic functions of the order ideals in the poset (\cite{Sta}). Choosing a maximal chain $\mathcal{C}$ identifies $\mathbb{R}^{R_{d,n}}$ with $\mathbb{R}^M$.
\par
We fix a maximal chain $\mathcal{C}$ and show that $\mathrm{NO}_{\mathcal C}(\ml)$ is the order polytope embedded in $\mathbb{R}^M$ by the above identification. By definition the order polytope coincides with $\Gamma_{\nu_{\Cc,\spec}}(R_1)$ hence contained in $\mathrm{NO}_{\mathcal C}(\ml)$, and the other inclusion is guaranteed by the Minkowski property of the Gelfand-Tsetlin polytopes.
\end{proof}

As a conclusion, we constructed for each maximal chain $\mathcal{C}$ in $I(d,n)$ a valuation $\mu_{\mathcal C,\spec}$ on $\mathbb C(\Gr_{d,n})$, such that the associated Newton-Okounkov body is unimodularly equivalent to the Gelfand-Tsetlin polytope. 

The subspace spanned by the standard monomials supported on $\mathcal C$ is a polynomial algebra. By the definition of the valuation $\nu_{\mathcal{C},\spec}$ in Section \ref{supportmap}, the image of this polynomial subalgebra under $\mu_{\mathcal{C},\spec}$ in the Newton-Okounkov body is clearly unimodularly equivalent to the standard simplex $\text{conv}(0,\mathbf{e}_1,\cdots,\mathbf{e}_N)$.

By taking the minimum of these valuations $\mu_{\mathcal{C},\spec}$ with respect to all maximal chains, we pass to a quasi-valuation $\mu_{\spec}$. According to the argument above, as well as Theorem \ref{Thm:Govern}, one obtains by varying maximal chains in $I(d,n)$ a triangulation of a Newton-Okounkov body such that the simplexes are parametrised by the maximal chains.

\section{Relation to Feigin-Fourier-Littelmann-Vinberg polytopes}\label{Sec:FFLV}

\subsection{Feigin-Fourier-Littelmann-Vinberg (FFLV) polytopes}

A \emph{Dyck path} in the poset $R_{d,n}$ is a chain $\mathbf{p}=\{\beta_1,\beta_2,\ldots,\beta_s\}$ in $R_{d,n}$ satisfying the following conditions:
\begin{enumerate}
\item $\beta_1=\alpha_{k,d}$ for $1\leq k\leq d$; $\beta_s=\alpha_{d,t}$ for $d\leq t\leq n-1$;
\item if $\beta_k=\alpha_{r,s}$ then $\beta_{k+1}$ is either $\alpha_{r,s+1}$ or $\alpha_{r+1,s}$.
\end{enumerate}
The set of all Dyck paths will be denoted by $\mathbb{D}_{d,n}$. 
\par
Let $(x_{i,j})_{1\leq i\leq d\leq j\leq n}$ be the coordinates in the real space $\mathbb{R}^M$ with $M=d(n-d)$. The FFLV polytope $\FFLV_{d,n}$ associated to $\Gr(d,n)$ (\cite{FFL1}) is the polytope in $\mathbb{R}^M$ defined by the following inequalities: for $1\leq i\leq d\leq j\leq n$,
$$x_{i,j}\geq 0,$$
$$\text{for any }\mathbf{p}\in\mathbb{D}_{d,n},\ \ \sum_{\alpha_{i,j}\in\mathbf{p}}x_{i,j}\leq 1.$$ 

These polytopes come from the study of PBW filtrations on Lie algebras \cite{FFL1}, which parametrise monomial bases of irreducible representations of $\SL_n$. These polytopes are identified in \cite{ABS} as marked chain polytopes, in the case of Grassmann varieties, $\FFLV_{d,n}$ is the chain polytope $\mathcal{C}(R_{d,n})$ of Stanley \cite{Sta} associated to the poset $R_{d,n}$.
\par
We provide a bijection between the set of order ideal $\mathcal{D}(J(\ml))$ in $J(\ml)$ and lattice points in $\FFLV_{d,n}$ by constructing for each order ideal a path partition of the corresponding sub-graph in $S_{d,n}$.
\par
For $\mi=[i_1,\ldots,i_d]\in\ml$, we want to define a path partition of $\spec(\mi)$. We first define the paring map 
$$p:\ml\backslash\{\mi_{0,0}\}\ra\ml,\ \  \mi\mapsto p(\mi)$$ 
as follows: Let $\mi=[1,\ldots,s,i_{s+1},\ldots,i_d]$ where $s\geq 0$ and $i_{s+1}\neq s+1$. The element $p(\mi)$ is defined to be $[1,\ldots,s,s+1,i_{s+1},\ldots,i_{d-1}]$.
\par
For any $\mi\in\ml\backslash\{\mi_{0,0}\}$, the paring map gives a sequence $\mi_0,\mi_1,\ldots,\mi_k$ where 
\begin{enumerate}
\item $\mi_0=\mi$, $\mi_k=\mi_{0,0}$ and $\mi_{k-1}\neq\mi_{0,0}$;
\item for any $s=1,\ldots,k$, $\mi_s=p(\mi_{s-1})$.
\end{enumerate}
The subsets $\spec(\mi_{s-1})\backslash\spec(\mi_{s})$ for $s=1,\ldots,k$ form a partition of $\spec(\mi)$. In the graph $S_{d,n}$, each part corresponds to a saturated Dyck path starting from the bottom row and end up with the rightmost column. They give a partition of $\mathcal{G}_{\spec(\mi)}$.
\par
For $s=1,\ldots,k$, we define $\beta_{s}:=\omega(\spec(\mi_{s-1})\backslash\spec(\mi_{s}))$ and $\beta(\mi)=\{\beta_1,\ldots,\beta_k\}$. We set $\beta(\mi_{0,0})=\emptyset$. 

\begin{proposition}
The following statements hold:
\begin{enumerate}
\item For $s=1,\ldots,k$, $\beta_{s}$ is a positive root in $R_{d,n}$;
\item The set $\beta(\mi)$ is an anti-chain in the root poset $R_{d,n}$.
\item The characteristic function $\chi_{\beta(\mi)}$ is a lattice point in $\FFLV_{d,n}$.
\item The lattice points in $\FFLV_{d,n}$ are $\{\chi_{\beta(\mi)}\mid \mi\in I(d,n)\}$.
\end{enumerate}
\end{proposition}

\begin{proof}
The statement (1) is clear by definition. 
\par
Anti-chains in $R_{d,n}$ are of the following form: $\{\alpha_{i_1,j_1},\alpha_{i_2,j_2},\ldots,\alpha_{i_s,j_s}\}$ where $i_1<i_2<\ldots<i_s$ and $j_1>j_2>\ldots>j_s$. If $\mi=[1,\ldots,s,i_{s+1},\ldots,i_d]$ where $s\geq 0$ and $i_{s+1}\neq s+1$, then $\omega(\spec(\mi)\backslash\spec(p(\mi)))=\ve_{i_d}-\ve_{s+1}=\alpha_{s+1,i_d-1}$. Therefore in the sequence $\{\beta_s=\alpha_{p_s,q_s}\mid s=1,\ldots,k\}$, $p_s<p_{s+1}$ and $q_s>q_{s+1}$, this proves (2).
\par
Recall that $\FFLV_{d,n}$ coincides with the chain polytope $\mathcal{C}(R_{d,n})$. By \cite[Theorem 2.2]{Sta}, characteristic functions of anti-chains are vertices in the chain polytope, the statement (3) is a consequence of (2).
\par
For weight reasons, the functions $\chi_{\beta(\mi)}$ are distinct. To show part (4), it suffices to prove that in the chain polytope $\mathcal{C}(R_{d,n})$, all lattice points are vertices, which is an easy counting.
\end{proof}

Let $V(\varpi_d)$ be the $d$-th fundamental representation of $\SL_n$ and for $1\leq i<j\leq n$, $f_{i,j}$ be a root vector corresponding to the negative root $-\alpha_{i,j}$. Each lattice point $\chi_A$ for $A=\{\alpha_{i_1,j_1},\alpha_{i_2,j_2},\ldots,\alpha_{i_s,j_s}\}$ in $\FFLV_{n,d}$ parametrizes a basis element $f_{i_s,j_s}\cdots f_{i_1,j_1}\cdot v_{\varpi_d}$ in $V(\varpi_d)$.

\begin{example}
We continue studying Examples \ref{Ex:I47} and \ref{Ex:I472}. Let $\mi=[2,4,5,7]\in I(4,7)$. The path partition of $\spec(\mi)$ induces a partition of $\mathcal{G}_{\spec(\mi)}$, which is given as follows:

\begin{equation}
\xymatrix{
& & &\ar@{=}[d]\alpha_6\\
&\ar@{=}[d]\alpha_3\ar@{=}[r]&\alpha_4\ar@{=}[r]&\alpha_5
\\
\alpha_1\ar@{=}[r]&\alpha_2& \alpha_3\ar@{=}[r]& \alpha_4.
\\
}
\end{equation}
In this case $\beta(\mi)=\{\alpha_{1,6},\alpha_{3,4}\}$. The corresponding characteristic function provides an element in $\FFLV_{4,7}$, which parametrises the basis element $f_{3,4}f_{1,6}\cdot v_{\varpi_4}\in V(\varpi_4)$ as a representation of $\SL_7$.
\end{example}

These results shed lights on the study of other Hodge algebra structures on $\mathbb{C}[\Gr_{d,n}]$, such that when a maximal chain is fixed, the associated toric variety is the one appeared in \cite{FFL2}, for details, see Section \ref{triang}.

\section{Remarks and outlooks}\label{Sec:Outlook}

\subsection{}
Besides join-irreducible elements, it is also possible to do the above construction using meet-irreducible elements. There exists a bijection $\eta$ between $I(d,n)$ and $I(n-d,n)$, sending a $d$-element subset of $\{1,2,\ldots,n\}$ to its complement.
\par
Let $\overline{I(n-d,n)}$ be the distributive lattice having the same elements as $I(n-d,n)$ whose join operation (resp. meet operation) is the meet operation (resp. join operation) in $I(n-d,n)$. Then $\eta:I(d,n)\ra \overline{I(n-d,n)}$ provides an isomorphism of distributive lattices. Meet-irreducible elements in $I(d,n)$ are meet-irreducible elements in $\overline{I(n-d,n)}$ hence join-irreducible elements in $I(n-d,n)$.
\par
As projective varieties, $\Gr_{d,n}$ is isomorphic to $\Gr_{n-d,n}$. Therefore the construction using meet-irreducible elements provides nothing new.

\subsection{}\label{triang}
One of the leading ideas of this paper is to get an interpretation and construction of standard monomial theory
(in the sense of Lakshmibai, Seshadri \textit{et al.} \cite{LR}) using filtrations obtained by valuations. 
Implicitly, the idea to use vanishing multiplicities to define and index standard monomials 
can be found already in \cite[Section 7]{LS},  in the filtration of the ideal sheaf associated
to $H(\tau)_{red}$, where $H(\tau)\subset X(\tau)$ is the zero set of the extremal weight section $p_\tau$
in the Schubert variety $X(\tau)\subset G/P\subset \mathbb P(V(\varpi))$ for a {\it classical type} fundamental weight 
flag variety. This geometric connection
also leads to the definition of LS-paths, see, for example, \cite{S2}. 
\par
Now in the case discussed in this paper, the fact that the degenerate algebra is a discrete Hogde algebra
implies that the semi-toric degeneration is a union of $\mathbb P^N$'s, which in turn implies that one gets 
a triangulation of the Newton-Okounkov body we started with. The latter is unimodularly equivalent to the
Gelfand-Tsetlin polytope (Proposition \ref{Prop:GT}). It is not expected, that this nice feature still holds in general. Indeed (see also Section \ref{LSALGEBRA}), it is expected that the standard monomial theory developed in \cite{L} will lead in the general case to a decomposition of the Newton-Okounkov body into the polytopes described in \cite{Deh}.

\subsection{}\label{Transfer}
In the setting of Stanley \cite{Sta}, one can associate to the poset $I(d,n)$ two polytopes, the order polytope and the chain polytope.
The first is realised in our setting as a Newton-Okounkov body, it is unimodularly equivalent to the Gelfand-Tsetlin polytope. 
The FFLV-polytope can also be realised as a Newton-Okounkov body \cite{FaFL1}. Now Stanley has described a piecewise
linear map between the two polytopes, which induces a bijection on the set of lattice points. It can be shown,
that the map restricted to the simplexes  (see Section~\ref{triang}) is an affine linear map, so the triangulation
of the Gelfand-Tsetlin polytope induces naturally a triangulation of the FFLV-polytope. It is expected, that this triangulation
has a similar standard monomial theory interpretation as in the Gelfand-Tsetlin case. Indeed, 
the results in \cite{HL} can be used to define a different Hodge algebra structure on the homogeneous
coordinate ring of $\Gr_{d,n}$, such that the construction described above leads to the FFLV-polytope as Newton-Okounkov body
and the triangulation induced by the discrete Hodge algebra is the image by the transfer map of the triangulation
of the Gelfand-Tsetlin polytope. It would be interesting to ``detropicalize'' Stanley's transfer map.

\subsection{}\label{LSALGEBRA}
As explained in Section \ref{triang}, to go from the case of Grassmann varieties to partial flag varieties, the Hodge algebra is needed to be upgraded to the LS-algebra to deal with the higher multiplicity phenomenon. Using LS-algebras, Chiriv\`{i} \cite{Chi} constructed semi-toric degenerations of partial flag varieties. In view of the construction in the current paper, it is natural to ask for a generalisation to partial flag varieties, that is to say, construct quasi-valuations to recover Chiriv\`{i} degenerations.

\subsection{}
Let $\ml$ be a distributive lattice and $\alpha,\beta:\ml\times\ml\ra\ml$ be two associative operations on $\ml$ satisfying: for any $\ell_1,\ell_2\in\ml$, $\alpha(\ell_1,\ell_2)\geq \beta(\ell_1,\ell_2)$. We say the pair $(\alpha,\beta)$ is compatible if in the straightening relation \eqref{straightlaw2} with $\ell_1,\ell_2$ non-comparable:
$$x_{\ell_1}x_{\ell_2}=\sum_{k_1\geq k_2}a_{k_1,k_2}x_{k_1}x_{k_2},$$
\begin{enumerate}
\item[(a)] the coefficient $a_{\alpha(\ell_1,\ell_2),\beta(\ell_1,\ell_2)}=1$;
\item[(b)] if for some $(k_1,k_2)\neq (\alpha(\ell_1,\ell_2),\beta(\ell_1,\ell_2))$, $a_{k_1,k_2}\neq 0$, then for every pair $(m_1,m_2)$ where $m_1\in\mSpec(\alpha(\ell_1,\ell_2))$ and $m_2\in\mSpec(\beta(\ell_1,\ell_2))$ such that $m_1\geq m_2$, one of the following statements holds:
\begin{itemize}
\item there exists $h\in\mSpec(k_1)$ such that $h>m_1$; 
\item the statement above does not hold, and there exist $h\neq h'\in\mSpec(k_1)$ such that $h=m_1$ and $h'>m_2$;
\item the statements above do not hold, and there exist $h\in\mSpec(k_1)$, $h'\in\mSpec(k_2)$ such that $h=m_1$ and $h'\geq m_2$.
\end{itemize}
\end{enumerate}

It would be interesting to study under what conditions on $\alpha$ and $\beta$, the pair $(\alpha,\beta)$ is compatible. The motivation of this question is to figure out how the FFLV polytopes for $\Gr(d,n)$ can be applied to construct semi-toric degenerations (see Section \ref{Transfer}), \emph{i.e.,} the FFLV polytopes appear as the Newton-Okounkov body when a maximal chain in $I(d,n)$ is fixed (see for example \cite{HL}).

\subsection{}
We finish this section by the following inverse problem: let $X\subset \mathbb{P}^N$ be a projective variety with homogeneous coordinate ring $\mathbb{C}[X]$. Let $\nu:\mathbb{C}[X]\ra\mathbb{Z}^N$ be a full rank valuation and $\mathrm{NO}_\nu(X)$ be the associated Newton-Okounkov body. Assume that $\mathrm{NO}_\nu(X)$ is a lattice polytope with a triangulation $\mathcal{T}$. Can one construct a Hodge algebra structure on $\mathbb{C}[X]$ such that the triangulation arising from the standard monomials coincides with $\mathcal{T}$?

\end{document}